\documentclass{amsart} 
\usepackage{amsmath,amsthm}
\usepackage{amssymb,amscd,latexsym}
\usepackage{boxedminipage}
\usepackage{eepic}
\usepackage{epic}
\usepackage{wrapfig}
\usepackage{graphicx}

\newcommand{\erase}[1]{}
\theoremstyle{remark}
\newtheorem{remark}{Remark} 
\newtheorem{theorem}{Theorem}[section]
\newtheorem{definition}[theorem]{Definition}
\newtheorem{proposition}[theorem]{Proposition}

\newtheorem{question}[theorem]{Question}

\newtheorem{lemma}[theorem]{Lemma}
\numberwithin{equation}{section}
\newcommand{\bp}{\begin{pmatrix}}
\newcommand{\ep}{\end{pmatrix}}
\newcommand{\bps}{\begin{smallmatrix}}
\newcommand{\eps}{\end{smallmatrix}}
\def\C{{\mathbb C}}

\def\Z{{\mathbb Z}}

\def \0{{\bf 0}}
\def \1{{\bf 1}}

\def \mf#1#2#3#4{
\xymatrix{{#1}\  \ar@<0.4ex>[r]^{{#2}} & \ {#4}
\ar@<0.4ex>[l]^{{#3}}
}
}

\def \mfs#1#2#3#4{\!
\xymatrix@C=1.5em{{#1} \! \ar@<0.2ex>[r]^{{#2}} & \! {#4}
\ar@<0.2ex>[l]^{{#3}}
}
\!}

\def \mfl#1#2#3#4{
\xymatrix@C=2.6em{{#1}\  \ar@<0.4ex>[r]^{{#2}} &\  {#4}
\ar@<0.2ex>[l]^{{#3}}
}
}

\def \mfss#1#2#3#4{\!
\xymatrix@C=1.5em{{#1} \ar@<0.3ex>[r]^{{#2}} & {#4}
\ar@<0.3ex>[l]^{{#3}}
}
\!}

\newcommand{\deeq}{\mathbin{\hbox{$=$ \lower 1.7pt\rlap{\hskip -8.5pt .}}}} 

\begin{document}
\title{The conjugacy problem for positive homogeneously presented monoids.} 

\author{Tadashi Ishibe}
\begin{abstract}
Let $M$ be a positive homogeneously presented monoid ${\langle L \mid R\,\rangle}_{mo}$. If $M$ satisfies the cancellation condition and carries certain particular elements similar to the \emph{fundamental elements} in Artin monoids, then the solvability of the conjugacy problem for $M$ implies that in the corresponding group ${\langle L \mid R\,\rangle}$. In addition to these conditions, if $M$ satisfies the LCM condition (i.e. any two elements $\alpha$ and $\beta$ in $M$ admit the left (resp.~right) least common multiple), then the solution to the conjugacy problem for $M$ is known. We will give two kinds of examples that do not satisfy only the LCM condition. For these examples, we will give a solution to the conjugacy problem by improving the method given by E. Brieskorn and K. Saito.
\end{abstract}

  







\maketitle

\section{Introduction} 
Let $M$ be a positive homogeneously presented monoid ${\langle L \mid R\,\rangle}_{mo}$. In this article, for the monoid $M$ that satisfies the cancellation condition and carries certain particular elements similar to the \emph{fundamental elements} in Artin monoids, we will discuss how to solve the conjugacy problem. In addition to these two conditions, if $M$ satisfies the LCM condition, the solution to the conjugacy problem has already been known (\cite{[B-S]}, \cite{[D-P]}, \cite{[P]}). We will deal with two kinds of examples  $G^+_{\mathrm{B_{ii}}}$ and $G^{+}_{m, n}$ that do not satisfy only the LCM condition. By improving the method in \cite{[B-S]}, we will give a solution to the conjugacy problem for them. As a consequence, the conjugacy problem in the corresponding groups can be solved.\par
 In \cite{[B-S]}, the notion of the \emph{Artin groups} was introduced as a generalization of the notion of the braid groups (\cite{[A1]}, \cite{[A2]}). In \cite{[De]}, the Artin groups were called by the alias of the \emph{generalized braid groups}. The Artin group appears as the fundamental group of the regular orbit space of a finite reflection group $W$ (\cite{[B]}). E. Brieskorn gave a special presentation of the fundamental group whose defining relations correspond to the finite Coxeter diagram of type $W$. The monoid defined by that presentation is called \emph{Artin monoid} (\cite{[B-S]}). In \cite{[B-S]}, by refering to the method in \cite{[G]}, they showed that the Artin monoid is \emph{cancellative} (i.e. $axb = ayb$ implies $x = y$ ) and that, for any two elements in the monoid, left (resp. right) common multiples exist. Hence, due to the \"Ore's criterion (see \cite{[C-P]}), the Artin monoid injects in the corresponding Artin group. Furthermore, they showed that the Artin monoid satisfies the LCM condition (see \cite{[B-S]} \S4). By using this property, in the monoid, they defined a particular element $\Delta$, the \emph{fundamental element}, as the least common multiple of all the generators. We note that the cancellativity and the existence of the least common multiple are the consequence of the \emph{reduction lemma} (see \cite{[B-S]} \S2). By using these properties, they proved that the solvability of the conjugacy problem in the Artin monoid implies that in the corresponding Artin group. Through the solution of the conjugacy problem for the Artin monoid, the conjugacy problem for the Artin group was solved. 
\par 
After this work, in the late $90's$, the notion of Artin group (resp. Artin monoid) is generalized by French mathematicians (\cite{[D-P]}, \cite{[D1]}), which is called the \emph{Garside group} (resp. \emph{Garside monoid}). The Garside group is defined as the group of fractions of a Garside monoid. A Garside monoid is a finitely generated monoid that satisfies the following conditions: $\mathrm{i})$ the monoid is cancellative; $\mathrm{ii})$ \emph{atomic} (i.e. the expressions of a given element have bounded lengths); $\mathrm{iii})$ the LCM condition is satisfied; $\mathrm{iv})$ a \emph{Garside element} exists. Hence, the Garside monoid trivially satisfies the \"Ore's criterion. In \S2, we will generalize the notion of fundamental element and, for simplicity, also call it a \emph{fundamental element}. Namely, we call an element $\Delta$ in a positive homogeneously presented monoid $M = {\langle L \mid R\,\rangle}_{mo}$ fundamental element, if there exists a permutation $\sigma$ of the set $L/\!\!\sim$ (: =the image of the set $L$ in $M$) such that 
\[
\Delta = s\cdot\Delta = \Delta\cdot \sigma(s)
\]
holds for all generators $s \in L/\!\!\sim$. 
\footnote
{In \S2, we will deal with a positively presented atomic monoid. Here we deal with a positive homogeneously presented monoid for simplicity.}
This notion does not always coincide with the least common multiple of all the generators. We note that, under the assumption that the monoid is cancellative, an element $\Delta$ in the monoid is a Garside element if and only if the element $\Delta$ is a fundamental element (Proposition 2.2). We will show that, for a Garside monoid, there is a unique fundamental element $\widehat{\Delta}$ that devides any fundamental elements in the monoid from the left and the right (Proposition 2.5). The set $\mathcal{F}(M)$ of all the fundamental elements of $M$ forms a subsemigroup of $M$ such that $\mathcal{QZ}(M)\mathcal{F}(M)\!=\!\mathcal{F}(M)\mathcal{QZ}(M)\!=\!\mathcal{F}(M)$, where $\mathcal{QZ}(M)$ is the set of all the quasi-central elements of $M$. \footnote
{An element $\Delta\!\in\!\! M$ is called quasi-central (\cite[\!7.1]{[B-S]})
if $s\!\cdot\!\Delta\!=\! \Delta\!\cdot\! \sigma(s)$ 
for $s\in L$.\!\!} From Proposition 2.5, we say that the two-sided idealistic subsemigroup $\mathcal{F}(M) ( \subseteq \mathcal{QZ}(M))$ is singly generated by $\widehat{\Delta}$. Moreover, in Proposition 2.8, we will show that any indecomposable quasi-central element $\Delta$ divides $\widehat{\Delta}$ from the left and the right (we call this perperty \emph{tame}). \footnote
{A positive homogeneously presented cancellative monoid that carries a fundamental element is called tame, if for any indecomposable quasi-central element $\Delta_0$ in the monoid there exists a minimal fundamental element $\Delta$ such that the element $\Delta_0$ divides $\Delta$ from the left and the right.} For a Garside group, the conjugacy problem can be solved (\cite{[P]}), by improving the method in \cite{[B-S]}. In \cite{[D-P]}, by showing a Garside group is a biautomatic group, they showed that the conjugacy problem is solvable. In \cite{[P]}, the author showed that the LCM condition implies the GCD condition (i.e. for arbitrary two elements in the monoid the left (resp. right) greatest common divisors exist). Then, due to the cancellativity and the LCM condition and the GCD condition, for an arbitrary element in a Garside group the author succeeded in constructing a normal form of it. As a consequence, the conjugacy problem in a Garside group is solved. In \S2, we will show that, in a positively presented atomic monoid that carries a fundamental element, the LCM condition is satisfied if and only if the GCD condition is satisfied (Proposition 2.4).\par
 Since the LCM condition is a strong assumption, some Zariski-van Kampen monoids do not satisfy the LCM condition (\cite{[B-M]}\cite{[I1]}\cite{[I2]}\cite{[S-I]}). As far as we know, for non-abelian positive homogeneously presented monoids that do not satisfy the condition $\mathrm{iii})$, there are few examples for which the cancellativity of them has been shown, since the pre-existing technique  is not perfect (\cite{[G]}\cite{[B-S]}\cite{[D2]}\cite{[D3]}). In the papers \cite{[I1]}\cite{[I2]}\cite{[S-I]}, we deal with non-abelian positive homogeneously presented monoids that do not satisfy only the condition $\mathrm{iii})$ and succeeded in showing cancellativity of them by improving the method. Besides the papers \cite{[I1]}\cite{[I2]}\cite{[S-I]}, for non-abelian positive homogeneously presented monoids that do not satisfy the condition $\mathrm{iii})$, there are known a few examples for which the cancellativity of them has been shown (e.g. \cite{[D4]}). We note that the cancellativity of these examples are shown by the pre-existing method (the consequence of the reduction lemma). In \cite{[I1]}, for the monoid, called the type $\mathrm{B_{ii}}$, that does not satisfy only the condition $\mathrm{iii})$, the author has solved the word problem and the conjugacy problem, and determined the center of it by showing the monoid injects in the corresponding group. In \cite{[I3]}, the author constructed other examples $G^{+}_{m, n}$ that do not satisfy only the condition $\mathrm{iii})$. Due to the cancellativity and the existence of fundamental elements, the word problem for them can be solved and the center of them was determined. As a consequence, the word problem in the corresponding groups can be solved and the center of them was determined. The conjugacy problem for them has not been solved yet since the former solution (\cite{[B-S]}, \cite{[D-P]}, \cite{[P]}) depends on the existence of the least common multiple. We need to extend the technique to solve the conjugacy problem. We remark that, in \cite{[I1]}, the author gave a solution to the conjugacy problem by using a very special type of normal form in the monoid. It is desirable that we give other solution that does not depend on the special context.\par
In this article, we will give a solution to the conjugacy problem in the monoids $G^{+}_{\mathrm{B_{ii}}}$ and $G^{+}_{m, n}$ by improving the method in \cite{[B-S]}. Let $M'$ be a positive homogeneously presented cancellative monoid that does not satisfy only the condition $\mathrm{iii})$. From the proposition 2.4, the GCD condtion for $M'$ is also not satisfied. Since neither the LCM condition nor the GCD condition is satisfied, we explore a solution to the conjugacy problem for $M'$ that depends on the reduction lemma. We remark that the reduction lemma knows the data for not only the cancellativity of the monoid but also the structure of the left (resp. right) minimal common multiples. For an element $w$ in a positive homogeneously presented cancellative monoid that carries a fundamental element, we call $A$ in the monoid a \emph{right transit element} of $w$, if the element $A$ is not $\varepsilon$ and there exists an element $Q$ such that an equation 
\[
AQ \deeq wA
\]
 holds. A right transit element $A$ of $w$ is called \emph{minimal}, if any right transit element of $w$ dividing $A$ from the left coincides with $A$ itself. The set of all right minimal transit elements of $w$ shall be denoted by $\mathrm{Trans^{min}}(w)$. Due to the homogeneity of the defining relations of the monoid, to solve the conjugacy problem for it, it is sufficient to show that, for an arbitrary element $w$ in it, we decide the set $\mathrm{Trans^{min}}(w)$ explicitly. In \S3, due to the reduction lemma for the monoid $G^{+}_{\mathrm{B_{ii}}}$, we will decide the set $\mathrm{Trans^{min}}(w)$ for an arbitrary element $w$ in it explicitly. In the papers \cite{[B-S]}, \cite{[P]}, the authors showed that for an arbitrary element $w$ in the monoid the set $\mathrm{Trans^{min}}(w)$ only consists of left divisors of the unique fundamental element $\widehat{\Delta}$ in the monoid, which will be called the property $(\mathrm{P}(w; \widehat{\Delta}))$ with respect to $w$. In \S4, due to the reduction lemma for the monoid $G^{+}_{m, n}$, we will show that, for an arbitrary element $w$ in the monoid, the property $(\mathrm{P}(w; \widehat{\Delta}))$ is satisfied. We remark that the idealistic subsemigroup $\mathcal{F}(G^{+}_{m, n}) (\subseteq \mathcal{QZ}(G^{+}_{m, n}))$ is singly generated and the monoid $G^{+}_{m, n}$ is tame. By generalizing this framework, we conjecture that, for a positive homogeneously presented cancellative tame monoid whose idealistic subsemigroup is finitely generated, the conjugacy problem in it is solvable.

\section{Positively Presented Monoid}

In this section, we first recall from \cite{[B-S]} some basic definitions and notations. Secondly, for a positive finitely presented group 
\[
G = \langle L \!\mid \!R\rangle, 
\]
we associate a monoid defined by it. We will extend a basic notion in \cite{[B-S]}, \emph{fundamental element}, for a positive presented atomic monoid. For a positively presented atomic cancellative monoid, we will show some convenient propositions.\footnote
{In this article, we only deal with the examples whose presentations are positive homogeneous. For some application, in this section we deal with a positively presented atomic monoid.} Lastly, by using a fundamental element $\Delta$ in the associated monoid, we will discuss the conjugacy problem in the group $G = \langle L \!\mid \!R\rangle$.\par
 Let $L$ be a finite set. Let $F(L)$ be the free group generated by $L$, and let $L^*$ be the free monoid generated by $L$ inside $F(L)$. We call the elements of $F(L)$ \emph{words} and the elements of $L^*$ \emph{positive words}. The empty word $\varepsilon$ is the identity element of $L^*$.
If two words $A$, $B$ are identical letter by letter, we write $A \equiv B$.
 Let $G = \langle L \!\mid \!R\rangle$ be a positively presented group (i.e. the set $R$ of relations consists of those of the form $R_i\!=\!S_i$ 
where 
$R_i$ and $S_i$ are positive words), where $R$ is the set of relations. We often denote the images of the letters and words under the quotient homomorphism 
$$\ F(L)\ \longrightarrow\ G$$
by the same symbols and the equivalence relation on elements $A$ and $B$ in $G$ is denoted by $A = B$.
 Secondly, we recall some terminologies and concepts on a monoid $M$. An element $U\!\in\! M$ is said to \emph{divide} $V\!\in\! M$ from the left (resp.~right), and denoted by $U|_lV$ (resp.~$U|_rV$), if there exists $W\!\in\! M$ 
such that $V\!=\! UW$ (resp.~$V\!=\! WU$). 
We also say that $V$ is \emph{left-divisible} (resp.\emph{right-divisible}) by $U$, or $V$ is 
a \emph{right-multiple} (resp.\emph{left-divisible}) of $U$. We say that $M$ satisfies \emph{the left {\rm (resp.} right{\rm )} 
LCM condition}, if for any two elements $U, V$ in $M$, there always exists their left (resp.~right) least common multiple. We say that $M$ satisfies \emph{the left {\rm (resp.} right{\rm )} GCD condition}, if for any two elements $U, V$ in $M$, there always exists their left (resp.~right) greatest common divisor. 
Lastly, we consider four operations on the set of subsets of a monoid $M$. For a subset $J$ of $M$, we put
\[
\mathrm{cm}_{r}(J) := \{ u \in M \mid j \,|_l \,u ,\, \forall j \in J \},\,\,\,\,\,\,\,\,\,\,\,\,\,\,\,\,\,\,\,\,\,\,\,\,\,\,\,\,
\]
\[
\,\mathrm{cd}_{l}(J) := \{ u \in M \mid u \,|_l \,j ,\, \forall j \in J \},\,\,\,\,\,\,\,\,\,\,\,\,\,\,\,\,\,\,\,\,\,\,\,\,\,\,\,
\]
\[
\mathrm{min}_{r}(J) := \{ u \in J \mid \exists v \in J \,\,\mathrm{s.t.}\,\, v \,|_l \,u \Rightarrow v = u \},
\]
\[
\mathrm{max}_{l}(J) := \{ u \in J \mid \exists v \in J \,\,\mathrm{s.t.}\,\, u \,|_l \,v \Rightarrow u = v \}.
\]
We note that $\mathrm{cm}_{r}(J)$, $\mathrm{min}_{r}(J)$, and $\mathrm{max}_{l}(J)$ may be the empty set. We consider their compositions by
\[
\mathrm{mcm}_{r}(J) := \mathrm{min}_{r}(\mathrm{cm}_{r}(J)),
\]
\[
\mathrm{mcd}_{l}(J) := \mathrm{max}_{l}(\mathrm{cd}_{l}(J)).
\]
In the same way, we define two operations $\mathrm{mcm}_{l}(J)$ and $\mathrm{mcd}_{r}(J)$. If $M$ satisfies the right LCM condition, then we should write $\mathrm{mcm}_{r}(J)$ by $\{ \mathrm{lcm}_{r}(J) \}$. If $M$ satisfies the left GCD condition, then we should write $\mathrm{mcd}_{l}(J)$ by $\{ \mathrm{gcd}_{l}(J) \}$.
Next, we recall from \cite{[S-I]}, \cite{[I1]}, \cite{[I3]} some terminologies and concepts on positively presented monoid. And we refer to some concepts from \cite{[D-P]}, \cite{[D1]}.\par 
\ \ \\
\begin{definition}
{\it Let 
$G = \langle L \!\mid \!R\rangle$
be a positively finitely presented group, where $L$ is the set of generators 
(called alphabet) and $R$ is the set of relations. 
Then we associate a monoid $G^+ = {\langle L \mid R\rangle}_{mo}$ defined as the quotient of the free monoid $L^*$ generated by $L$ by the equivalence relation  defined as follows:  \par 
$\mathrm{i})$ two words $U$ and $V$ in $L^*$ are called \emph{elementarily 
equivalent} if either $U \equiv V$ or $V$ is obtained from $U$ by substituting 
a substring $R_i$ of $U$ by $S_i$ where $R_i\!=\!S_i$ is a relation of $R$ 
($S_i = R_i$ is also a relation if $R_i=S_i$ is a relation), \par 
$\mathrm{ii})$ two words $U$ and $V$ in $L^*$ are called \emph{equivalent}, denoted by $U \deeq V$, if there exists 
a finite sequence $U\! \equiv \!W_0, W_1,\ldots, W_n\! \equiv \!V $ of words in $L^*$ for $n\!\in\!\Z_{\ge0}$
such that $W_i$ is elementarily equivalent to $W_{i-1}$ for $i=1,\ldots, n$.\par1. We say that $G^+$ is \emph{atomic}, if there exists a map:
\[
\nu\ : \ G^+\ \longrightarrow\ \Z_{\ge0}
\]
such that $\mathrm{i})$ $\nu(\alpha) = 0$\,$\Longleftrightarrow$\,$\alpha = 1$ and $\mathrm{ii})$ an inequality:
\[
\nu(\alpha\beta) \geq \nu(\alpha) + \nu(\beta)
\]
is satisfied for any $\alpha, \beta \in G^{+}$. If $G^+ = {\langle L \mid R\rangle}_{mo}$ is a positive homogeneously presented monoid (i.e. the set $R$ of relations consists of those of the form $R_i = S_i$ where $R_i$ and $S_i$ are positive words of the same length ), it is clear that $G^+$ is an atomic monoid. An element $\alpha \not= 1$ in $G^{+}$ is called an \emph{atom} if it is indecomposable, namely, $\alpha = \beta \gamma$ implies $\beta = 1$ or $\gamma = 1$.\par
2. We suppose that $G^{+}$ satisfies the condition of atomic monoid. Here, we write the set of generators $L$ by $\{ g_1, g_2, \ldots, g_{m} \}$. If, for some positive word\\
 $w(g_1, \ldots, g_{i-1}, g_{i+1}, \ldots, g_m)$ (i.e. a word that is written by the generators except $g_i$), $g_i = w(g_1, \ldots, g_{i-1}, g_{i+1}, \ldots, g_m)$ is a relation of $R$, then we call the generator $g_i$ a \emph{dummy generator}. We note that, in the set $R$, a relation that has a form of $g_i = w(g_1, \ldots, g_i, \ldots, g_m)$ must be the form $g_i = w(g_1, \ldots, g_{i-1}, g_{i+1}, \ldots, g_m)$ or a trivial form $g_i = g_i$, because we suppose here that $G^{+}$ is an atomic monoid. We denote by $L'$ the set of all dummy generators of the monoid $G^+$. We put $\widetilde{L} := L \setminus L'$. We remark that the set $\widetilde{L}$ may be the empty set. We note that, if $G^+$ is an atomic monoid, the image of the set $\widetilde{L}$ in $G^{+}$ is equal to the set of all the atoms.\par 
3. We say that $G^+$ is \emph{cancellative}, if 
an equality $AXB \deeq \! AYB$  for\\
 $A, B, X, Y \!\in G^+$ implies $X \deeq \! Y$.\par

4. The natural homomorphism $\pi:  G^+\to G$ will be called the \emph{localization homomorphism}. 

5. An element $\Delta \in G^+$ is called \emph{quasi-central} (also see [B-S] 7.1), if there exists a permutation $\sigma$ of $\widetilde{L}$  such that 
\[
s\cdot\Delta \deeq  \Delta\cdot \sigma(s)
\]
holds for all generators $s \in \widetilde{L}$. 
 The set of all quasi-central elements is denoted by $\mathcal{QZ}(G^+)$. We note that, if the monoid $G^{+}$ is a cancellative monoid, there exists a unique permutation $\sigma_\Delta$ for a quasi-central element $\Delta$. The order of an element $\sigma_{\Delta}$ in the permutation group $\mathfrak{S}(\widetilde{L})$ is denoted by $\mathrm{ord}(\sigma_{\Delta})$. Note that $\Delta^{\mathrm{ord}(\sigma_{\Delta})}$ belongs to the center $\mathcal{Z}(G^+)$ of the monoid $G^+$ and $\varepsilon \in \mathcal{QZ}(G^+)$.\par
6. An element $\Delta$$\in G^+$ is called a \emph{Garside element}
if the sets of left- and right-divisors of $\Delta$ coincide, generate $G^{+}$, and are finite in number.

7. An element $\Delta$ in $G^+$ is called a \emph{fundamental element}
if there exists a permutation $\sigma$ of $\widetilde{L}$ such that, for any $s \in \widetilde{L}$, there exists 
$\Delta_s$$\in G^+$ satisfying the following relation: 
\[
\Delta \deeq s \cdot \Delta_s \deeq  \Delta_s \cdot \sigma(s).
\]
We write the set of all fundamental elements of $G^+$ by $\mathcal{F}(G^+)$. Note that $\varepsilon \not\in \mathcal{F}(G^+)$. If the monoid $G^{+}$ is a cancellative monoid, there exists a unique permutation $\sigma_\Delta$ for a fundamental element $\Delta$. Note that $\Delta^{\mathrm{ord}(\sigma_{\Delta})}$ belongs to the center $\mathcal{Z}(G^+)$ of the monoid $G^+$. It is easy to show that
\[
\mathcal{F}(G^+)\mathcal{QZ}(G^+)=\mathcal{QZ}(G^+) \mathcal{F}(G^+)= \mathcal{F}(G^+).
\]

8. A fundamental element $\Delta$ is called \emph{minimal}, if any fundamental element dividing $\Delta$ from right or left coincides with $\Delta$ itself.

9. A quasi-central element $\Delta$ is called  \emph{indecomposable}, if it does not decompose into a product of two non-trivial quasi-central elements. We note that the identity element $\varepsilon$ in not indecomposable.}


\end{definition}
\begin{remark}{\it  Let $G^+$ be a positively presented atomic monoid ${\langle L \mid R\,\rangle}_{mo}$ that satisfies the cancellation condition. Let $\Delta$ be a non-trivial quasi-central element. We put
\[
L(\Delta) := \{ s \in \widetilde{L} \mid s \,|_l \,\Delta  \}.
\]
Then, we have an equality
\[
\forall s \in L(\Delta),\,\, \Delta \deeq s\cdot \Delta_s \deeq \Delta_s \cdot \sigma_{\Delta}(s).
\]
For each $s$ in $L(\Delta)$, the quotient can be uniquely determined in the monoid $G^+$. Here, we write it by $\Delta_s$. From this property, we easily show that the sets of left- and right-divisors of $\Delta$ coincide.}
\end{remark}
\ \ \\
\begin{remark}{\it Let $G^+$ be a positively presented atomic monoid. The two-sided idealistic subsemigroup $\mathcal{F}(G^+) (\subseteq \mathcal{QZ}(G^+))$ is not finitely generated in general ( See \cite{[I1]} Remark 5.9 ). }
\end{remark}
\ \ \\
From the definitions, it follows that the notion of fundamental elements is equivalent to the notion of Garside elements. 
\ \ \\
\begin{proposition}
{\it  Let $G^+$ be a positively presented atomic monoid ${\langle L \mid R\,\rangle}_{mo}$ that satisfies the cancellation condition. Then:  \par 
 (1)  An element $\Delta$ in $G^+$ is a fundamental element if and only if $\Delta$ is a Garside element. \par 
 (2)  Let $\Delta_1$ and $\Delta_2$ be quasi-central elements in $G^+$ that satisfies a relation: there exists an element $d$ in $G^+$ such that $\Delta_{1} d \deeq \Delta_2$. Then, the element $d$ is a quasi-central element in $G^+$\\}
\end{proposition}
\begin{proof} 
See \cite{[I3]} \S2.
\end{proof}
\ \ \\
\begin{proposition}
{\it  Let $G^+$ be a positively presented atomic cancellative monoid\\
 ${\langle L \mid R\,\rangle}_{mo}$ that carries a fundamental element $\Delta$ and satisfies the condition that any two letters $\alpha, \beta$ in $\widetilde{L}$ admit the left (resp.~right) least common multiple. \\
Then, the monoid $G^+$ satisfies the left (resp.~right) LCM condition.}
\end{proposition}
\begin{proof} 
Let $\Delta \in \mathcal{F}(G^+)$ be a fundamental element. Due to the property of fundamental element, for any $\gamma \in G^+$, there exists a sufficiently large integer $l$ such that $\gamma$ divides $\Delta^{l}$ from the left and
 the right. Therefore, we say that, for any two elements $U, V$ in $G^+$, $\mathrm{mcm}_{r}(\{ U, V \}) \not=\emptyset$. We consider the following equation
\[
U X \deeq V Y.
\]
Since the monoid $G^+$ is cancellative and, for any two letters $\alpha, \beta$ in $\widetilde{L}$, they admit the right least common multiple, 
we solve the equation uniquely, if the solution exists. The existence of the solution is guaranteed since, for any two elements $U, V$ in $G^+$, $\mathrm{mcm}_{r}(\{ U, V \}) \not=\emptyset$. We conclude that the monoid $G^+$ satisfies the right LCM condition. In the same way, we also conclude that the monoid $G^+$ satisfies the left LCM condition.
\end{proof}
\ \ \\
\begin{proposition}
{\it  Let $G^+$ be a positively presented atomic monoid ${\langle L \mid R\,\rangle}_{mo}$ that carries a fundamental element $\Delta$.\\
Then, the monoid $G^+$ satisfies the left (resp.~right) LCM condition if and only if the monoid $G^+$ satisfies the right (resp.~left)  GCD condition.}
\end{proposition}
\begin{proof} 
Assume that the monoid $G^+$ satisfies the left LCM condition. We easily show that, for any two elements $\alpha, \beta$ in $G^+$, the element $\mathrm{lcm}_{l}(\mathrm{cd}_{r}(\{ \alpha, \beta \}))$ is the right greatest common divisor for the set $\{ \alpha, \beta \}$.\\
Next we assume the monoid $G^+$ satisfies the right GCD condition. We suppose that the monoid $G^+$ does not satisfy the left LCM condition. Then, there exist two elements $\alpha, \beta$ in $G^+$ such that the set $\mathrm{mcm}_{l}(\{ \alpha, \beta \})$ does not consist of a single element. Thanks to the existence of a fundamental element $\Delta$, we say that $\mathrm{mcm}_{l}(\{ \alpha, \beta \}) \not=\emptyset$. Therefore, we can take two different elements  $C_1, C_2$ from the set $\mathrm{mcm}_{l}(\{ \alpha, \beta \})$. We say that both $\alpha$ and $\beta$ are left divisors of $\mathrm{gcd}_{r}(\{ C_1, C_2 \})$. Hence, we say that
\[
 C_1 = \mathrm{gcd}_{r}(\{ C_1, C_2 \}) = C_2 .
\]
We have a contradiction.
\end{proof}
\ \ \\
\begin{proposition}
{\it  Let $G^+$ be a positively presented atomic monoid ${\langle L \mid R\,\rangle}_{mo}$ that satisfies the cancellation condition. We suppose that the monoid $G^+$ satisfies the LCM condition and $\mathcal{F}(G^+) \not=\emptyset$ (i.e. the monoid $G^+$ is a Garside monoid).\\
Then, the idealistic subsemigroup $\mathcal{F}(G^+) (\subseteq \mathcal{QZ}(G^+))$ is singly generated.}
\end{proposition}
\begin{proof} 
We take arbitrary two elements $\Delta_1$, $\Delta_2$ in $\mathcal{F}(G^+) $. Due to the Proposition 2.4, we say that, for two elements $\Delta_1$, $\Delta_2$, there exists the right (resp.~left) greatest common divisor. Since both $\Delta_1$ and $\Delta_2$ are Garside elements, we easily show that
\[
\mathrm{gcd}_{r}(\{ \Delta_1, \Delta_2  \}) = \mathrm{gcd}_{l}(\{ \Delta_1, \Delta_2 \}).
\]
We put $\delta := \mathrm{gcd}_{r}(\{ \Delta_1, \Delta_2  \}) = \mathrm{gcd}_{l}(\{ \Delta_1, \Delta_2 \})$. We easily show that the element $\delta$ is a Garside element. We also say that the element $\delta$ is a fundamental element. Hence, there exists a unique element $\widehat{\Delta}$ such that $\mathcal{F}(G^+) ( \subseteq \mathcal{QZ}(G^+))$ is singly generated by $\widehat{\Delta}$.
\end{proof}
Thanks to the Proposition 2.5, in a Garside monoid we can choose the unique minimal fundamental element $\widehat{\Delta}$.
\begin{lemma}
{\it  Let $G^+$ be a positively presented atomic monoid ${\langle L \mid R\,\rangle}_{mo}$ that satisfies the cancellation condition and carries a fundamental element. Let $\delta$ be an element that satisfies the condition that there exists a mapping $\sigma: \widetilde{L} \to G^+$
\[
\forall s \in \widetilde{L},\,\,  s \cdot \delta \deeq \delta \cdot \sigma(s).
\] 
Then, the element $\delta$ is a quasi-central element.}
\end{lemma}
\begin{proof} 
Due to the property of fundamental elements, there exists a fundamental element $\Delta$ such that 
\[
\Delta \deeq \delta \cdot D.
\]
In addition to this property, we can assume that, for any letter $s$ in the set $\widetilde{L}$, the letter $s$ devides $D$ from the left and the right. Since the monoid $G^+$ is a cancellative monoid, there exists a unique permutation $\sigma_\Delta$ for the fundamental element $\Delta$. From the assumption, we have an equation
\[
\forall s \in \widetilde{L},\,\, s\cdot \Delta \deeq \delta \cdot \sigma(s) \cdot D \deeq \delta \cdot D \cdot  \sigma_{\Delta}(s).
\]
Due to the cancellativity, we have an equation
\begin{equation}
\forall s \in \widetilde{L},\,\, \sigma(s) \cdot D \deeq D \cdot  \sigma_{\Delta}(s).
\end{equation}
We put
\[
L(\delta) := \{ s \in \widetilde{L} \mid s \,|_l \,\delta  \}.
\]
By using the equation (2.1), we also have an equation
\begin{equation}
\forall s \in L(\delta),\,\, s\cdot \delta_s \deeq \delta_s \cdot \sigma(s).
\end{equation}
For each $s$ in $L(\delta)$, the quotient can be uniquely determined in the monoid $G^+$. Here, we write it by $\delta_s$. We put
\[
L_{reg}(\delta) := \{ s \in L(\delta) \mid \forall k \in \Z_{\ge0},\,\, \sigma^{k}(s) \in \widetilde{L} \}, L_{irreg}(\delta) := L(\delta)\setminus L_{reg}(\delta).
\]
Since the set $\widetilde{L}$ is a finite set, the restriction $\sigma|_{L_{reg}(\delta)}$ is a permutation.
\\
{\bf Claim.} $L_{reg}(\delta) \not=\emptyset$.
\begin{proof} We assume that $L_{reg}(\delta) =\emptyset$. For any element $s$ in $L(\delta)$, there exists a positive integer $k(s)$ such that 
\[
\sigma^{k(s)}(s) \in G^{+}\setminus \widetilde{L}. 
\]
We write the element $\delta$ by a product of atoms
\[
\delta \deeq \alpha_1 \cdot \alpha_2 \cdots \alpha_l .
\]
From the equation (2.2), we have an equation
\[
\alpha_1 \cdot \alpha_2 \cdots \alpha_l \deeq \sigma(\alpha_1) \cdot \sigma(\alpha_2) \cdots \sigma(\alpha_l)
\]
\[
\cdots \deeq \sigma^{N}(\alpha_1) \cdot \sigma^{N}(\alpha_2) \cdots \sigma^{N}(\alpha_l).\,\,\,\,\,\,\,\,\,\,
\]
Hence, expressions of the element $\delta$ by product of atoms do not have bounded lengths. Since the monoid $G^+$ is an atomic monoid, we have a contradiction. 
\end{proof}
We assume that $L_{irreg}(\delta) \not=\emptyset$. We put
\[
L'_{irreg}(\delta) := \{ s \in L_{irreg}(\delta) \mid  \sigma(s) \in G^{+}\setminus \widetilde{L} \}.
\]
We consider the following two cases.\par
Case 1: $\exists s_0 \in L'_{irreg}(\delta)$ s.t. the expressions of the element $\sigma(s_0)$ consist of atoms in $L_{reg}(\delta)$.\\
We write the element $\sigma(s_0)$ by a product of atoms in $L_{reg}(\delta)$
\[
\sigma(s_0) \deeq \beta_1 \cdot \beta_2 \cdots \beta_k.
\]
Then, there exist atoms $\beta'_{i}$ in $L_{reg}(\delta)$ $(i= 1, \ldots, k)$ such that $\sigma(\beta'_{i}) \deeq \beta_i$$(i= 1, \ldots, k)$. Therefore, we have an equation
\begin{equation}
\delta \deeq \beta'_{k} \cdot \delta_{\beta'_{k}} \deeq \delta_{\beta'_{k}} \cdot \beta_k \deeq \delta_{s_{0}}\cdot \beta_1 \cdot \beta_2 \cdots \beta_k.
\end{equation}
Due to the cancellativity, we have an equality
\[
\delta_{\beta'_{k}} \deeq \delta_{s_{0}}\cdot \beta_1 \cdot \beta_2 \cdots \beta_{k-1}.
\]
By substituting $\delta_{\beta'_{k}}$ by $\delta_{s_{0}}\cdot \beta_1 \cdot \beta_2 \cdots \beta_{k-1}$ in the equation (2.3), we have
\[
s_0 \cdot \delta_{s_0} \deeq \delta \deeq \beta'_{k} \cdot \delta_{s_{0}}\cdot \beta_1 \cdot \beta_2 \cdots \beta_{k-1} \deeq \beta'_1 \cdot \beta'_2 \cdots \beta'_{k} \cdot \delta_{s_{0}}.
\]
Then, we have an equality $s_0 \deeq \beta'_1 \cdot \beta'_2 \cdots \beta'_{k}$. A contradiction.\par
Case 2: $\forall s \in L'_{irreg}(\delta)$, the expressions of the element $\sigma(s)$ by product of atoms contain an atom in $L_{irreg}(\delta)$.\\
We consider an equation again
\[
\delta \deeq \alpha_1 \cdot \alpha_2 \cdots \alpha_l \deeq \sigma(\alpha_1) \cdot \sigma(\alpha_2) \cdots \sigma(\alpha_l)
\]
\[
\cdots \deeq \sigma^{N}(\alpha_1) \cdot \sigma^{N}(\alpha_2) \cdots \sigma^{N}(\alpha_l).\,\,\,\,\,\,\,\,\,\,\,\,\,\,\,\,\,\,\,\,
\]
Since the monoid $G^+$ is an atomic monoid, we have a contradiction. \\

Therefore, we conclude that $L(\delta) = L_{reg}(\delta)$. Next, we put
\[
L(D) := \{ s \in \widetilde{L} \mid s \,|_r \,D  \}, L_{reg}(D) := \{ s \in L(D) \mid \forall k \in \Z_{\ge0},\,\, \varphi^{k}(s) \in \widetilde{L} \}
\]
, where $\varphi(s) := \sigma(\sigma^{-1}_{\Delta}(s))$.
In the same way, we conclude that $L(D) = L_{reg}(D)$. From the assumption, we say that  $L(D) = \widetilde{L}$. Hence, the element $D$ is a fundamental element. We conclude that $\delta$ is a quasi-central element.
\end{proof}
\begin{lemma}
{\it We suppose that a positively presented monoid $G^+ = {\langle L \mid R\,\rangle}_{mo}$ is a Garside monoid. Let $\Delta_1$ and $\Delta_2$ be quasi-central elements in $G^+$. \\
Then, the element $\mathrm{gcd}_{l}(\{ \Delta_1, \Delta_2  \}) \deeq \mathrm{gcd}_{r}(\{ \Delta_1, \Delta_2  \})$ is a quasi-central element.}
\end{lemma}
\begin{proof} 
Since the monoid satisfies the LCM condition, the GCD condition is also satisfied. We put $\delta_{l} := \mathrm{gcd}_{l}(\{ \Delta_1, \Delta_2  \})$ and $\delta_{r} := \mathrm{gcd}_{r}(\{ \Delta_1, \Delta_2  \})$. Dut to the Remark 1, we easily show that $\delta_{l} \deeq \delta_{r}$. We write it by $\delta$. If $\delta$ is equal to $\Delta_1$ or $\Delta_2$, then $\delta$ is a quasi-central element. Therefore, we suppose that $\delta \not= \Delta_{1}, \Delta_2$. If the element $\delta$ satisfies the following condition
\[
\forall s \in \widetilde{L} ,\,\, \mathrm{lcm}_{r}(\{ s \cdot \delta, \delta \}) = \{ s \cdot \delta\},
\]
due to the Lemma 2.6, we say that there exists a map $\sigma : \widetilde{L} \to \widetilde{L}$ such that
\[
\forall s \in \widetilde{L} ,\,\, s \cdot \delta \deeq \delta \cdot \sigma(s).
\]
Since the monoid $G^+$ is a cancellative monoid, we say that the map $\sigma$ is a permutation of $\widetilde{L}$. Therefore, we say that $\delta$ is a quasi-central element. If the element $\delta$ does not satisfy the above condition, then we say that there exists a letter $s_0 \in \widetilde{L}$ such that there exists a unique element $d \not= \varepsilon$ such that 
\[
\mathrm{lcm}_{r}(\{ s_0 \cdot \delta, \delta \}) = \{ s_0 \cdot \delta d\}.
\]
The uniqueness of $d$ is guaranteed by the LCM condition. We define two elements $X_1, X_2$ by 
\[
\Delta_{1} \deeq \delta \cdot X_1, \Delta_{2} \deeq \delta \cdot X_{2}.
\]
Since the monoid $G^+$ is a cancellative monoid, we say that two elements $X_1$ and $X_2$ are determined uniquely. We easily show that $d |_l X_1$ and $d |_l X_2$. Hence, we say that $\delta \cdot d \in \mathrm{cd}_{l}(\{ \Delta_1, \Delta_2  \})$. We have a contradiction. Therefore, we conclude that the element $\delta$ is a quasi-central element.
\end{proof}
As a consequence of Proposition 2.7, we prove the following proposition.
\begin{proposition}
{\it  We suppose that a positively presented monoid $G^+ = {\langle L \mid R\,\rangle}_{mo}$ is a Garside monoid. Let $\widehat{\Delta}$ be the unique minimal fundamental element in $G^+$, and let $\Delta$ be an indecomposable quasi-central element in $G^+$. \\
Then, the element $\Delta$ divides $\widehat{\Delta}$ from the left and the right.}
\end{proposition}
\ \ \\
\begin{definition}
{\it Let $G^+$ be a positively presented atomic cancellative monoid ${\langle L \mid R\,\rangle}_{mo}$ that carries a fundamental element. \\
The monoid $G^+$ is called  \emph{tame} if for any indecomposable quasi-central element $\Delta_0$ in $G^+$ there exists a minimal fundamental element $\Delta$ such that the element $\Delta_0$ divides $\Delta$ from the left and the right. }
\end{definition}
\begin{remark}{\it In \cite{[I1]}, the author studied some positive homogeneously presented cancellative monoid. The author made a list of all the minimal fundamental elements and indecomposable quasi-central elements in the monoid. The monoid is not tame.}
\end{remark}
Lastly, we discuss the word and conjugacy problem in a positive homogeneously presented group.
\begin{definition}
{\it Let $G^+ = {\langle L \mid R\rangle}_{mo}$ be a positive homogeneously presented monoid.

1) 
For arbitrary two words $U$, $V$ in $L^*$, give an algorithm 
that decides whether $U \deeq V$ in $G^+$ 
 or not.
 
2) 
For arbitrary two words $U$, $V$ in $L^*$, give an algorithm 
that decides whether there exists an element $A$ in $G^+$ 
 such that $AU \deeq VA$ (then we write 
   $U$$\underset{mo}{\sim}$$V$) or not. \par 
The problems 1), 2) are called the
\emph{word  problem} and the \emph{conjugacy problem} 
in a monoid $G^+$, respectively.}
\end{definition}
\begin{remark}{\it If $\mathcal{F}(G^+)\not=\emptyset$, then the relation $\underset{mo}{\sim}$ is an equivalence relation.}
\end{remark}
\begin{lemma}
{\it  Let $G = \langle L \!\mid \!R\rangle$ be a positive homogeneously presented group, and let 
$G^+ = {\langle L \mid R\rangle}_{mo}$ be the associated monoid. Assume that the monoid $G^+$ is a cancellative monoid and $\mathcal{F}(G^+) \not=\emptyset$. Then:  \par 
 
 (1)  The localization homomorphism $\pi:  G^+\to G$ is injective. \par 
 (2)  The word problem in $G$ is solvable.\par
 (3)  The conjugacy problem in $G^+$ is solvable if and only if the conjugacy problem in $G$ is solvable.}
\end{lemma}

\begin{proof}(1)  Let $\Delta \in \mathcal{F}(G^+)$ be a fundamental element. We can easily show that, for any $U \in G^+$, there exists a sufficiently large integer $l$ such that $U$ divides $\Delta^{l}$ from the left and the right. Hence, we show that the monoid $G^+$ satisfies \"Ore's condition (see \cite{[C-P]}). Therefore, the localization homomorphism $\pi$ is injective.\par
(2)  We put $\Lambda : = \Delta^{\mathrm{ord}(\sigma_{\Delta})}$, which belongs to the center $\mathcal{Z}(G^+)$ of the monoid $G^+$. 
For any two elements $U, V$ in $G$, there exists a non-negative integer $k$ in $\Z_{\ge0}$ such that both ${(\pi(\Lambda))}^{k} U$ and ${(\pi(\Lambda))}^{k} V$ are equivalent to positive words. Since the localization homomorphism $\pi$ is injective, there exists a unique element $U' \in G^+$ (resp. $V' \in G^+$) such that
\[
\pi(U')={(\pi(\Lambda))}^{k} U (\rm{resp}. \pi(V')={(\pi(\Lambda))}^{k} V).
\] 
 Therefore, we can show that $U = V$ can be shown in $G$ algorithmically if and only if $U' \deeq  V'$ can be shown in $G^+$ algorithmically. Because the monoid $G^+$ is an atomic monoid, we can obtain algorithmically all the possible expressions of two words  $U'$ and $V'$ in $G^+$ in a finite number of steps. Hence, by comparing two types of complete lists of all the possible expressions of words $U'$ and $V'$, we decide in a finite number of steps whether  $U' \deeq V'$ or not. Consequently, the word problem in $G$ can be solved.\par
(3)  If two elements $U$ and $V$ in $G$ are conjugate, then there exists a word $B$ such that $BU = VB$. There exists a non-negative integer $l$ in $\Z_{\ge0}$ such that ${(\pi(\Lambda))}^{l} B$ is equivalent to a positive word. Since ${\pi(\Lambda)}$ belongs to the center of the group $G$, we say that two elements $U$ and $V$ in  $G$ are conjugate precisely when there is a positive word $A$ such that $AU$ is equivalent to $VA$. Therefore, due to the injectivity of the localization homomorphism $\pi$, we can show that the conjugacy problem in $G^+$ is solvable if and only if the conjugacy problem in $G$ is solvable.
\end{proof}

\section{How to solve the conjugacy problem}

Let $G^+$ be a positive homogeneously presented cancellative monoid ${\langle L \mid R\,\rangle}_{mo}$ that carries a fundamental element $\Delta$. In this section, we will discuss how to solve the conjugacy problem for $G^+$ by referring to the method given by E. Brieskorn and K. Saito in \cite{[B-S]}. \\
For an arbitrary element $w \in G^+$, we put
\[
\mathrm{Conj}^{+}(w) := \{ V \in G^+ \mid AV \deeq wA,\, A \in G^+ \}.
\]
If we give an algorithm to decide the set $\mathrm{Conj}^{+}(w)$ for an arbitrary element $w \in G^+$, then we can say that the conjugacy problem for it is solvable. To decide the set $\mathrm{Conj}^{+}(w)$ for an arbitrary element $w \in G^+$, we introduce the following set
\[
\mathrm{O}^{(1)}(w; \Delta) := \{ V \in G^+ \mid AV \deeq wA,\, A\, |_l \Delta \}.\]
Since the set of left divisors of $\Delta$ is a finite set, we say that one can decide the set $\mathrm{O}^{(1)}(w; \Delta)$ algorithmically. Since the set $\mathrm{O}^{(1)}(w; \Delta)$ is a finite set, one can iterate the construction and obtain the sets
\[
\mathrm{O}^{(k+1)}(w; \Delta) := \{ V \in G^+ \mid AV \deeq UA,\,U \in \mathrm{O}^{(k)}(w; \Delta),\, A\, |_l \Delta \}.
\]
We easily show that the sets $\mathrm{O}^{(k)}(w; \Delta)$ can be decided algorithmically. We have $\mathrm{O}^{(k)}(w; \Delta)\subseteq \mathrm{O}^{(k+1)}(w; \Delta)$. Due to the homogeneity, we say that all elements in $\mathrm{O}^{(k)}(w; \Delta)$ have the same length. Hence, there exists a positive integer $k_0$ such that
\[
\mathrm{O}^{(k_0)}(w; \Delta) = \mathrm{O}^{(k_{0}+1)}(w; \Delta) = \cdots.
\]
We put
\[
\mathrm{O}(w; \Delta) := \mathrm{O}^{(k_0)}(w; \Delta).
\]
From the construction, we have $\mathrm{O}(w; \Delta) \subseteq \mathrm{Conj}^{+}(w)$. 

\begin{remark}{\it If an equality $\mathrm{O}(w; \Delta) = \mathrm{Conj}^{+}(w)$ holds for an arbitrary element $w \in G^+$, then the conjugacy problem for it is solvable.}
\end{remark}
\begin{definition}
{\it Let $G^+ = {\langle L \mid R\rangle}_{mo}$ be a positive homogeneously presented cancellative monoid that carries a fundamental element $\Delta$, and let $w$ be an element in $G^+$. \\

1. An element $A$ in $G^+$ is called a \emph{right transit element} of $w$ in $G^+$, if\\
$\mathrm{i})$  $A \not= \varepsilon$.\\
$\mathrm{ii})$  There exists an element $Q$ in $G^+$ such that an equation $AQ \deeq wA$ holds.\\
The set of all right transit elements of $w$ in $G^+$ shall be denoted by $\mathrm{Trans}(w)$. \par
2. A right transit element $A$ of $w$ is called \emph{minimal}, if any right transit element of $w$ dividing $A$ from the left coincides with $A$ itself.\\
The set of all right minimal transit elements of $w$ in $G^+$ shall be denoted by $\mathrm{Trans^{min}}(w)$.}
\end{definition}

\begin{definition}
{\it Let $G^+ = {\langle L \mid R\rangle}_{mo}$ be a positive homogeneously presented cancellative monoid that carries a fundamental element $\Delta$, and let $w$ be an element in $G^+$. \\
We introduce a property with respect to $w$:\\
$(\mathrm{P}(w; \Delta))$: For any element $A$ in $\mathrm{Trans}(w)$, there exists an element $\delta(A)$ in $\mathrm{Trans}(w)$ such that $\delta(A) \,|_l A$ and $\delta(A)\,|_l \Delta$ hold.
}
\end{definition}

\begin{lemma}
{\it  Let $G^+ = {\langle L \mid R\rangle}_{mo}$ be a positive homogeneously presented cancellative monoid that carries a fundamental element $\Delta$, and let $w$ be an element in $G^+$.\\
If the property $(\mathrm{P}(w; \Delta))$ with respect to $w$ is satisfied,
then an equality 
\[
\mathrm{O}(w; \Delta) = \mathrm{Conj}^{+}(w)
\]
 holds.}
\end{lemma}
\begin{proof}
We take an element $V$ in $\mathrm{Conj}^{+}(w)$. By definition, there exists an element $A$ in $G^+$ such that an equality $AV \deeq wA$ holds. Since the property $(\mathrm{P}(w; \Delta))$ is satisfied, there exists a sequence $(\delta_{1}, \delta_{2}, \ldots, \delta_{k})$ of elements in $\mathrm{Trans}(w)$ such that each $\delta_{i}$ devides $\Delta$ $(i=1,\ldots, k)$ from the left and we have a decomposition
\[
A \deeq \delta_{1} \cdot \delta_{2} \cdots \delta_{k}.
\]
Hence, we have $V \in \mathrm{O}(w; \Delta)$.
\end{proof}

\begin{theorem}
{\it  Let $G^+ = {\langle L \mid R\rangle}_{mo}$ be a positive homogeneously presented cancellative monoid that carries a fundamental element $\Delta$ and satisfies the LCM condition.\\
Then, the conjugacy problem for $G^+$ is solvable.}
\end{theorem}

\begin{proof}
We take an arbitrary element $w$ in $G^+$. We verify the property $(\mathrm{P}(w; \Delta))$ with respect to $w$. We take an element $A$ in $\mathrm{Trans}(w)$.
By definition, there exists an element $Q$ in $G^+$ such that an equation $AQ \deeq wA$ holds. Since the element $A$ is not $\varepsilon$, there exists a letter $l_0$ such that the letter $l_0$ devides $A$ from the left. Then, there exists a unique element $d_1$ such that
\[
\mathrm{lcm}_{r}(\{ l_0 , w \cdot l_{0} \}) = \{ w \cdot d_{1} \}.
\]
Thanks to the property of fundamental elements, we say that the element $d_{1}$ devides $\Delta$ from the left. We rewrite $d_0 = l_0$. One can iterate the construction and obtain a sequence $(d_0, d_1, d_2, \ldots)$ of elements in $G^+$. Namely, $d_i$ and $d_{i+1}$ satisfy the relation $\mathrm{lcm}_{r}(\{ d_i , w \cdot d_{i} \}) = \{ w \cdot d_{i+1} \}$ for $i=0, 1, \ldots$. We easily show that the element $d_{i}$ devides $d_{i+1}$ from the left for $i=0, 1, \ldots$. Thanks to the property of fundamental elements, we also say that the element $d_{i}$ devides $\Delta$ from the left for $i=0, 1, \ldots$. Due to the homogeneity of the defining relations in the monoid $G^+$, there exists a positive integer $k_0$ such that
\[
d_{k_0} \deeq d_{k_{0}+1} \deeq \cdots.
\]
Then, the element $d_{k_0}$ belongs to the set $\mathrm{Trans}(w)$. Thanks to the Lemma 3.3, we say that the conjugacy problem for $G^+$ is solvable.
\end{proof}
\begin{remark}{\it The element $d_{k_0}$ is not a right minimal transit element of $w$ in general.}
\end{remark}
Though there exists an element $w_0$ in $G^+$ such that the property $(\mathrm{P}(w_{0}; \Delta))$ is not satisfied, there is a possibility for solving the conjugacy problem for $G^+$. Due to the homogeneity, we show that for an arbitrary element $w$ in $G^+$ the set $\mathrm{Conj}^{+}(w)$ is a finite set. To solve the conjugacy problem for $G^+$, it is sufficient to show that for an arbitrary element $w$ in $G^+$ one can decide the set $\mathrm{Trans^{min}}(w)$ explicitly.\\\ \ \\
{\bf Example.}\,\,We recall an example, the monoid $G^+_{\mathrm{B_{ii}}}$, from \cite{[I1]}. The monoid $G^+_{\mathrm{B_{ii}}}$ has the following presentation 

\[
\begin{array}{lll}
G^+_{\mathrm{B_{ii}}}:=
\biggl{\langle}
a,b,c\,
\biggl{|}
\begin{array}{lll}cbb=bba,\\
 ab=bc,\\
 ac=ca
  \end{array}
\biggl{\rangle}_{mo} .
\end{array}
\]
We recall a lemma from \cite{[S-I]} \S7.

\begin{lemma}
{\it  Let $X$ and $Y$ be two words in $G^{+}_{\mathrm{B_{ii}}}$ of length $r\in \Z_{\ge0}$.
\smallskip
\\
{\rm (i)}\, If $vX \deeq \! vY$ for some $v \in \{a, b, c \}$, then $X \deeq \! Y$.\\
{\rm (ii)}\, If $a X \deeq b Y$, then $X \deeq b Z$, $Y \deeq c Z$ for some positive word $Z$.\\
{\rm (iii)}\, If $a X \deeq c Y$, then $X \deeq c Z$, $Y \deeq a Z$ for some positive word $Z$.\\
 {\rm (iv)}\, If $b X \deeq c Y$, then there exist an integer $k \in \Z_{\ge0}$ and a positive word $Z$ such that $X \deeq c^{k}ba \cdot Z$, $Y \deeq a^{k}bb \cdot Z$.}
 \par 
\end{lemma}
Thanks to the Lemma 3.5, we say that the monoid $G^+_{\mathrm{B_{ii}}}$ is a left cancellative monoid. In the monoid $G^+_{\mathrm{B_{ii}}}$, we have an anti-homomorphism 
 $\varphi:\mathrm{G^+_{B_{ii}}}\rightarrow\mathrm{G^+_{B_{ii}}}$,
 $W\mapsto \varphi(W):=\sigma$$(rev(W))$, where $\sigma$ is a
 permutation $\big(^{a\,\, b\,\, c}_{c\,\, b\,\, a}\big)$ and $rev(W)$
 is the reverse of the word $W=x_1 x_2 \cdots x_t$ ($x_i$ is a letter or
 an inverse of a letter) given by the word  $x_t  \cdots x_2 x_1$. If $\beta \alpha \deeq \gamma \alpha$, then $\varphi(\beta \alpha) \deeq \varphi(\gamma \alpha)$, i.e.,  $\varphi(\alpha) \varphi(\beta) \deeq \varphi(\alpha)\varphi(\gamma)$. Using the left cancellation condition, we obtain $\varphi(\beta) \deeq \varphi(\gamma)$ and, hence, $\beta \deeq \gamma $.\par
 From the Lemma 3.5 ${\rm (iv)}$, we say that in the monoid $G^+_{\mathrm{B_{ii}}}$ the LCM condition is not satisfied. Nevertheless, we can decide the set $\mathrm{Trans^{min}}(w)$ explicitly for an arbitrary element $w$ in $G^+_{\mathrm{B_{ii}}}$.\par
The monoid $G^+_{\mathrm{B_{ii}}}$ carries a unique indecomposable quasi-central element \\ $\Delta_0 := bbb$ that is not a fundamental element and infinite minimal fundamental elements $\Delta_k := (bc^{k})^3$. If $\Delta'$ is an indecomposable quasi-central element, then there exists a non-negative integer $k$ in $\Z_{\ge0}$ such that $\Delta'$ is equivalent to $\Delta_k$. By using the element $\Delta_0$, we can introduce a special normal form in the monoid $G^+_{\mathrm{B_{ii}}}$. Since both sides of the defining relations of the monoid $G^+_{\mathrm{B_{ii}}}$ contain the same number of the letter $b$, for an arbitrary element $W$ in the monoid $G^+_{\mathrm{B_{ii}}}$, the number of the letter $b$ in $W$ ought to be preserved in the process of rewriting $W$.\par
 For each $j \in \Z_{\ge0}$, let
\[
W(j) : =\{w \in G^+_{\mathrm{B_{ii}}} \mid w\,\, \mathrm{contains\,\, the\,\, letter\,\,} b\,\, \mathrm{just}\,\, j\mathrm{\mathchar`-times} \}.
\]
We recall two facts from \cite{[I1]} \S5.
\begin{proposition}
{\it If $w \in W(j)$ $(j \ge4)$, then $\Delta_0 |_l w$ and $\Delta_0 |_r w$.}
\end{proposition}
\noindent
Therefore, for an arbitrary element $w$ in $G^+_{\mathrm{B_{ii}}}$, we define a non-negative integer
\[
k(w):= \mathrm{max}\{ k \in  \Z_{\ge0} \mid \Delta^{k}_{0}\,|_l w \}.
\]
\begin{proposition}
{\it  Let $w$ be an element in the monoid $G^+_{\mathrm{B_{ii}}}$. 
We write \\
$w \deeq \Delta^{k(w)}_{0} \cdot w_{\mathrm{remain}}$. Let $j$ be the number of the letter $b$ in $w_{\mathrm{remain}}$.\\
Then, $w_{\mathrm{remain}}$ has the following normal form:\\
\noindent
\,\,\,\,$j = 0$\,: \,\,\,\, $w_{\mathrm{remain}} \deeq a^p c^q$  \,\,\,\,\,\,\,\,\,\,\,\,\,\,\,\,\,\,\,\,\,\,\,$(p,q\in \Z_{\ge0})$ \par\noindent 
\,\,\,\,$j = 1$\,: \,\,\,\, $w_{\mathrm{remain}} \deeq a^p c^q b a^r$\,\,\,\,\,\,\,\,\,\,\,\,$(p,q,r\in \Z_{\ge0})$
 \par\noindent 
\,\,\,\,$j = 2$\,: \,\,\,\, $w_{\mathrm{remain}} \deeq a^p c^q bb c^r$\,\,\,\,\,\,\,\,\,\,$(p,q,r\in \Z_{\ge0})$
 \par\noindent 
\,\,\,\,$j = 3$\,: \,\,\,\, $w_{\mathrm{remain}} \deeq a^p c^q b a^r bb$\,\,\,\,\,\,\,$(p,q,r\in \Z_{\ge0})$}
\end{proposition}
\noindent
The element $\Delta_0$ belongs to the center $\mathcal{Z}(G^+_{\mathrm{B_{ii}}})$ of the monoid $G^+_{\mathrm{B_{ii}}}$. To solve the conjugacy problem for $G^+_{\mathrm{B_{ii}}}$, it is sufficient to show that, for an arbitrary element $w \in W(j)$ $(j \le3)$, we decide the set $\mathrm{Trans^{min}}(w)$ explicitly. Here, we only show the results. The calculations by using the Lemma 3.5 are written in \cite{[I4]}. We consider the following four cases.\\
{\bf Case 1.}  Let $w$ be an element in $W(0)$.\\
\noindent
\,\,\,\,$p = 0, q \ge1 $\,: \,\,\,\,$\mathrm{Trans^{min}}(w) = \{ a, c, bb \}$.   \,\,\,\,\,\,\,\,\,\,\,\,\,\,\,\,\,\,\,\,\,\,\, \par\noindent 
\,\,\,\,$q = 0, p \ge1 $\,: \,\,\,\,$\mathrm{Trans^{min}}(w) = \{ a, b, c \}$. \,\,\,\,\,\,\,\,\,\,\,\,
 \par\noindent 
\,\,\,\,$p, q \ge1 $\,: \,\,\,\,\,\,\,\,\,\,\,\,\,\,\,$\mathrm{Trans^{min}}(w) = \{ a, c \} \cup \{ bbc^{i}b \mid i = 0, 1, \ldots \}$. \,\,\,\,\,\,\,\,\,\,
 \par\noindent 
Since $\Delta_i \deeq bc^{i}bc^{i}bc^{i} \deeq bbc^{i}ba^{i}c^{i}$, we say that  for any element $u$ in $\mathrm{Trans^{min}}(w)$ there exists a non-negative integer $i \in \Z_{\ge0}$ such that $u$ devides $\Delta_i$ from the left.
\par\noindent
{\bf Case 2.}  Let $w$ be an element in $W(1)$.\\
\noindent
\,\,\,\,$p = 0, q \ge1, r\ge0 $\,: \,\,\,\,$\mathrm{Trans^{min}}(w) = \{ b, c \}$.   \,\,\,\,\,\,\,\,\,\,\,\,\,\,\,\,\,\,\,\,\,\,\, \par\noindent 
\,\,\,\,$q = 0, p \ge1,  r\ge0$\,: \,\,\,\,$\mathrm{Trans^{min}}(w) = \{ a, b, cba, \ldots, c^{r}ba^{r} \}$. \,\,\,\,\,\,\,\,\,\,\,\,
 \par\noindent 
\,\,\,\,$p, q \ge1, r\ge0$\,: \,\,\,\,\,\,\,\,\,\,\,\,\,\,\,$\mathrm{Trans^{min}}(w) = \{ a, b, c \}$. \,\,\,\,\,\,\,\,\,\,
 \par\noindent 
\,\,\,\,$p = q = 0,  r\ge0$\,: \,$\mathrm{Trans^{min}}(w) = \{ b, cba, \ldots, c^{r}ba^{r} \} \cup \{ a^{i}c^{i}ba^{i} \mid i = r+1, \ldots \}.$ \,\,\,\,\,\,\,\,\,\,
 \par\noindent 
Since $\Delta_i \deeq bc^{i}bc^{i}bc^{i} \deeq c^{i}ba^{i}bba^{i} \deeq a^{i}c^{i}ba^{i}bb$, we say that  for any element $u$ in $\mathrm{Trans^{min}}(w)$ there exists a non-negative integer $i \in \Z_{\ge0}$ such that $u$ devides $\Delta_i$ from the left.
\par\noindent
{\bf Case 3.}  Let $w$ be an element in $W(2)$.\\
\noindent
\,\,\,\,$p = 0, q \ge1, r\ge0 $\,: \,\,\,\,$\mathrm{Trans^{min}}(w) = \{ b, c \}$.   \,\,\,\,\,\,\,\,\,\,\,\,\,\,\,\,\,\,\,\,\,\,\, \par\noindent 
\,\,\,\,$q = 0, p \ge1, r\ge0$\,: \,\,\,\,$\mathrm{Trans^{min}}(w) = \{ a, b \}$. \,\,\,\,\,\,\,\,\,\,\,\,
 \par\noindent 
\,\,\,\,$p, q \ge1, r\ge0$\,: \,\,\,\,\,\,\,\,\,\,\,\,\,\,\,$\mathrm{Trans^{min}}(w) = \{ a, b, c \}$
. \,\,\,\,\,\,\,\,\,\,
 \par\noindent 
\,\,\,\,$p = q = 0, r\ge0$\,: \,$\mathrm{Trans^{min}}(w) = \{ b, acba, a^{2}c^{2}ba^{2}, \ldots \}.$ \,\,\,\,\,\,\,\,\,\,
\par\noindent
Since $\Delta_i \deeq bc^{i}bc^{i}bc^{i} \deeq a^{i}c^{i}ba^{i}bb$, we say that  for any element $u$ in $\mathrm{Trans^{min}}(w)$ there exists a non-negative integer $i \in \Z_{\ge0}$ such that $u$ devides $\Delta_i$ from the left.\par
\noindent
{\bf Case 4.}  Let $w$ be an element in $W(3)$.\\
We have $\mathrm{Trans^{min}}(w) = \{ a, b, c \}$.\\
For an arbitrary element $w$ in $G^+_{\mathrm{B_{ii}}}$, we easily show that, for any element $u$ in the set $\mathrm{Trans^{min}}(w)$, there exists a non-negative integer $i \in \Z_{\ge0}$ such that $u$ devides $\Delta_i$ from the left.\par
From the results on the monoid $G^+_{\mathrm{B_{ii}}}$, we may say that there is a relation between the set $\mathrm{Trans^{min}}(w)$ and the set of all indecomposable quasi-central elements and all minimal fundamental elements.
\begin{question}{\it Let $G^+ = {\langle L \mid R\rangle}_{mo}$ be a positive homogeneously presented cancellative monoid that carries a fundamental element, and let $w$ be an element in $G^+$. Furthermore, we suppose that the monoid $G^+$ is tame and the idealistic subsemigroup $\mathcal{F}(G^+) (\subseteq \mathcal{QZ}(G^+))$ is finitely generated by $\Delta_1,\ldots, \Delta_k$. We ask whether, for any element $u$ in the set $\mathrm{Trans^{min}}(w)$, there exists a minimal fundamental element $\Delta_i$ such that $u$ devides $\Delta_i$ from the left and the right. Then, from the set $\{ \Delta_1, \ldots, \Delta_k \}$, we construct a fundamental element $\widehat{\Delta}$ algorithmically such that each $\Delta_j$$( j= 1, \ldots, k )$ devides $\widehat{\Delta}$ from the left and the right. It comes to the conclusion that an equality 
\[
\mathrm{O}(w; \widehat{\Delta}) = \mathrm{Conj}^{+}(w)
\]
 holds.}
\end{question} 
\ \ \\
{\bf Example.}\,\, In \cite{[D4]}, the author investigated the monoids $M_1, M_2$ and $M_3$. All of them are positive homogeneously presented cancellative monoids that carry fundamental elements. We can easily show that they are tame and their idealistic subsemigroups are finitely generated. In each example, we can check the property $(\mathrm{P}(w; \widehat{\Delta}))$ for an arbitrary element $w$ in the monoid. It comes to the conclusion that the equality $\mathrm{O}(w; \widehat{\Delta}) = \mathrm{Conj}^{+}(w)$ holds.

\section{The conjugacy problem for the monoid $G_{m, n}^{+}$}
In this section, we will deal with the examples $G_{m, n}^{+}$ (\cite{[I3]}). For the monoids $G_{m, n}^{+}$, we will show that the idealistic subsemigroup is singly generated and they are tame monoids. However, they do not satisfy the LCM condition. Nevertheless, in the monoids $G_{m, n}^{+}$, we will show that the conjugacy problem can be solved by verifying the property $(\mathrm{P}(w; \Delta))$ for an arbitrary element $w$ in them.\par
First, we recall an example, the monoid $G^+_{m, n}$ ($m,n \in \{ 2, 3, \ldots \}$), from \cite{[I3]}. The author studied the presentation of the fundamental group of the complement of certain complexified real affine line arrangement and associated a monoid defined by it. The monoid has the following presentation:\par
\[
\begin{array}{rlll}
\Biggl{\langle}\!
s,t_1, \ldots, t_m, u_1, \ldots, u_n
\biggl{|}
\begin{array}{lll}
 \left[ s, t_1, \ldots, t_m \right], \left[ s, u_1, \ldots, u_n \right], \\
\left[ t_i, u_j \right] (i=1, \ldots, m, j=1, \ldots, n) \\
  \end{array}\
\Biggl{\rangle}_{mo}\ ,
\end{array}\
\] 
where a symbol $\left[ x_{i_1}, x_{i_2}, \ldots, x_{i_k} \right]$ denotes the cyclic relations:
\[
x_{i_1} x_{i_2} \cdots x_{i_k} = x_{i_2} \cdots x_{i_k}x_{i_1} = x_{i_k}x_{i_1} \cdots x_{i_{k-1}}.
\]
The monoid $G^+_{m, n}$ denotes the associated monoid. Before continuing further, we recall some notation from \cite{[I3]} \S4. We put
\[
\Delta := s \cdot t_1 \cdots t_m \cdot u_1 \cdots u_n,\,\,\Delta_1 := s \cdot t_1 \cdots t_m,\,\, \Delta_2 := s \cdot u_1 \cdots u_n,
\]
\[
I_1 := \{1, \ldots, m \},\,\, I_2 :=\{1, \ldots, n \},
\]
\[
L_{0} := \{\,s, t_1, \ldots, t_m, u_1, \ldots, u_n\, \}, L_1 := \{\, t_1, \ldots, t_m \, \},\,\, L_2 :=\{\, u_1, \ldots, u_n\, \},
\]
\[
F^{+}_1 := F^{+}(\underline{t}),\,\, F^{+}_2 := F^{+}(\underline{u}),
\]
\[
F^{+}_{1, \mathrm{rm}} := \{ w(\underline{t}) \in  F^{+}_1 \mid  (t_1 \cdots t_m) \not|_r w(\underline{t}) \},
\]
\[
F^{+}_{2, \mathrm{rm}} := \{ w(\underline{u}) \in  F^{+}_2 \mid  (u_1 \cdots u_n) \not|_r w(\underline{u}) \},
\]
\[
F^{+}_{1, \mathrm{cons}} := \{ w \in  F^{+}_1 \mid  \exists i_0, j_0 \in I_1 (i_0 \leq j_0)\,\, \mathrm{s.t.}\, w = t_{i_0}t_{i_0 + 1} \cdots t_{j_0}\},
\]
\[
F^{+}_{2, \mathrm{cons}} := \{ w \in  F^{+}_2 \mid  \exists i_0, j_0 \in I_2 (i_0 \leq j_0)\,\, \mathrm{s.t.}\, w = u_{i_0}u_{i_0 + 1} \cdots u_{j_0}\}.
\]
 For arbitrary element $w(\underline{t})$ in $F^{+}_{1}$ and $w(\underline{u})$ in $F^{+}_{2}$, we put
\[
\mathrm{Div}_1(w(\underline{t})) := \{ w \in  F^{+}_{1, \mathrm{cons}}  \mid  w \,|_r \,w(\underline{t}) \},
\]
\[
\mathrm{Div}_2(w(\underline{u})) := \{ w \in  F^{+}_{2
, \mathrm{cons}}  \mid  w \,|_r \,w(\underline{u}) \}.
\]
We remark that there exists a unique element $w_{0, 1}$ in $\mathrm{Div}_1(w(\underline{t}))$ (resp. $w_{0, 2}$ in $\mathrm{Div}_2(w(\underline{u}))$ ) such that $w_1 \,|_r \,w_{0, 1}$ for any element $w_1$ in $\mathrm{Div}_1(w(\underline{t}))$ (resp. $w_2 \,|_r \,w_{0, 2}$ for any element $w_2$ in $\mathrm{Div}_2(w(\underline{u}))$ ). We put
\[
\mathrm{C}_{1}(w(\underline{t})):= w_{0,1}, \mathrm{C}_{2}(w(\underline{u})):= w_{0,2}.
\]
In view of the defining relations of $G^+_{m, n}$, there exists an element  $w'(\underline{t})$ in $F^{+}_1$ (resp. $w'(\underline{u})$ in $F^{+}_2$) such that we have a decomposition $w(\underline{t}) \equiv w'(\underline{t}) \mathrm{C}_{1}(w(\underline{t}))$ (resp. $w(\underline{u}) \equiv w'(\underline{u}) \mathrm{C}_{2}(w(\underline{u})))$ in $G^+_{m, n}$. We put
\[
\mathrm{R}_{1}(w(\underline{t})):= w'(\underline{t}), \mathrm{R}_{2}(w(\underline{u})):= w'(\underline{u}).
\]
 For arbitrary left divisor $v_{1}$ of $\Delta_1$ (resp. $v_2$ of $\Delta_2$), the quotient can be uniquely determined in the monoid $G^+_{m, n}$ respectively. We denote it by $\Delta_{1, v_1}$ (resp. $\Delta_{2, v_2}$).\\
We recall a lemma from \cite{[I3]} \S4.
\begin{lemma}
{\it Let $X$ and $Y$ be positive words in $G^+_{m, n}$ of length $r\in \Z_{\ge0}$ and let $Y^{(h)}$ be a positive word in $G^+_{m, n}$ of length $h \in \{\,0, \ldots, r \}$.
\smallskip
\\
{\rm (i)}\, If $vX \deeq \! vY$ for some $v \in L_0$, then $X \deeq \! Y$.\\
{\rm (ii)}\, If $t_{i} X \deeq u_{j}Y$ $(t_i \in L_1, u_j \in L_2)$, then $X \deeq u_{j} Z$, $Y \deeq t_{i} Z$ for some positive word $Z$.\\
{\rm (iii)}\, If $sX \deeq w(\underline{t}) Y^{(h)}$ for some positive word $w(\underline{t})$ of length $r-h+1$ in $F^{+}_{1}$, then $X \deeq \Delta_{1, s} \cdot \mathrm{R}_{1}(w(\underline{t})) \cdot Z$, $Y^{(h)} \deeq \Delta_{1, \mathrm{C}_{1}(w(\underline{t}))} \cdot Z$ for some positive word $Z$.\\
{\rm (iv)}\, If $sX \deeq w(\underline{u}) Y^{(h)}$ for some positive word $w(\underline{u})$ of length $r-h+1$ in $F^{+}_{2}$, then $X \deeq \Delta_{2, s} \cdot \mathrm{R}_{2}(w(\underline{u})) \cdot Z$, $Y^{(h)} \deeq \Delta_{2, \mathrm{C}_{2}(w(\underline{u}))} \cdot Z$ for some positive word $Z$.\\
{\rm (v)}\, If $t_{i}X \deeq w(\underline{t}) Y^{(h)}$ for some $t_i$ in $L_1$ and some positive word $w(\underline{t})$ of length $r-h+1$ in $F^{+}_{1}$ that satisfies $t_i \,\not|_l \,w(\underline{t})$, then there exists a word $w(\underline{u})$ in $F^{+}_{2, \mathrm{rm}}$ such that $X \deeq w(\underline{u}) \cdot \Delta_{1, t_{i}} \cdot \mathrm{R}_{1}(w(\underline{t})) \cdot Z$, $Y^{(h)} \deeq w(\underline{u}) \cdot \Delta_{1, \mathrm{C}_{1}(w(\underline{t}))} \cdot Z$ for some positive word $Z$.\\
{\rm (vi)}\, If $u_{i}X \deeq w(\underline{u}) Y^{(h)}$ for some $u_i$ in $L_2$ and some positive word $w(\underline{u})$ of length $r-h+1$ in $F^{+}_{2}$ that satisfies $u_i \,\not|_l \,w(\underline{u})$, then there exists a word $w(\underline{t})$ in $F^{+}_{1, \mathrm{rm}}$ such that $X \deeq w(\underline{t}) \cdot \Delta_{2, u_{i}} \cdot \mathrm{R}_{2}(w(\underline{u})) \cdot Z$, $Y^{(h)} \deeq w(\underline{t}) \cdot \Delta_{2, \mathrm{C}_{2}(w(\underline{u}))} \cdot Z$ for some positive word $Z$.}\\
\end{lemma}
Thanks to the Lemma 4.1, we say that the monoid $G^+_{m, n}$ is a left cancellative monoid. In the monoid $G^+_{m, n}$, we have an anti-homomorphism $\varphi:G^{+}_{m, n}\rightarrow G^{+}_{m, n}$,
 $W\mapsto \varphi(W):=\sigma$$(rev(W))$, where $\sigma$ is a
 permutation $\big(^{\,s\,\, t_1\,\,\, \cdots \,\, t_m\,\,u_1\,\,\,\cdots \,\,u_n}_{\,s\,\, t_m\,\cdots \,\,\, t_1\,\,\,u_n\,\,\cdots\,\,\, u_1}\big)$ and $rev(W)$
 is the reverse of the word $W=x_1 x_2 \cdots x_k$ ($x_i$ is a letter) given by the word  $x_k  \cdots x_2 x_1$. By a similar argument in Example \S3, we show that the monoid $G^{+}_{m, n}$ is a right cancellative monoid. \\
We easily show that $\Delta$ is a fundamental element. We recall two facts from \cite{[I3]} \S5.
\begin{proposition}
{\it The element $\Delta$ is a unique minimal fundamental element and $\mathcal{F}(G^+_{m, n}) = \mathcal{QZ}(G^+_{m, n}) \setminus \{ \varepsilon \}$.}
\end{proposition}
\begin{proposition}
{\it $\mathrm{mcm}_{r}(L_{0}) = \mathrm{mcm}_{l}(L_{0}) = \{ \Delta \}$.}
\end{proposition}
Thanks to the Proposition 4.2, we say that the idealistic subsemigroup \\
$\mathcal{F}(G^+_{m, n}) (\subseteq \mathcal{QZ}(G^+_{m, n}))$ is singly generated by $\Delta$ and the monoid $G^+_{m, n}$ is tame.\par
 From the Lemma 4.1 ${\rm (v)}$ and ${\rm (vi)}$ , we say that in the monoid $G^+_{m, n}$ the LCM condition is not satisfied. Nevertheless, we can verify the property $(\mathrm{P}(w; \Delta))$ for an arbitrary element $w$ in the monoid $G^+_{m, n}$.\par
We have an important remark on the monoid $G_{m, n}^{+}$.
\begin{remark}{\it  For each letter $v$ in $L_0$, both sides of the defining relations of $G^{+}_{m, n}$ contain the same number of the letter $v$. For arbitrary word $W$ in $G^{+}_{m, n}$, the number of the letter $v$ in $W$ ought to be preserved in the process of rewriting $W$.}
\end{remark}
We consider the following set
\[
\mathcal{W}_{m, n} := \{ w \in G^{+}_{m, n} \mid w \,\,\mathrm{does}\,\mathrm{not}\,\mathrm{contain}\,\Delta_{1}\, (\mathrm{resp.}\, \Delta_{2})\, \mathrm{as}\,\mathrm{substring}  \}.
\]

\begin{proposition}
{\it Let $w$ be an element in the set $G^{+}_{m, n}$. If neither $\Delta_{1, s}$ nor $\Delta_{2, s}$  devides $w$ from the left, then the element $w$ belongs to the set $\mathcal{W}_{m, n}$.}
\end{proposition}
By definition, in the process of rewriting $w$ in $\mathcal{W}_{m, n}$, we only use the defining relations $\left[ t_i, u_j \right] (i=1, \ldots, m, j=1, \ldots, n)$ in $G^{+}_{m, n}$. 
\begin{proposition}
{\it Let $w$ be an element in $\mathcal{W}_{m, n}$. Then, the element $w$ has the following normal form:\\
the element $w$ can be uniquely written in the form
\[
w \deeq w_{0}(\underline{t})\cdot w_{0}(\underline{u}) \cdot s \cdot w_{1}(\underline{t})\cdot w_{1}(\underline{u}) \cdot s \cdots s \cdot w_{N}(\underline{t})\cdot w_{N}(\underline{u}) ,
\]
where $w_{0}(\underline{t}), w_{1}(\underline{t}), \ldots,$ and $w_{N}(\underline{t})$ are some words in $F^{+}_{1}$ and $w_{0}(\underline{u}), w_{1}(\underline{u}), \ldots$ and $w_{N}(\underline{u})$ are some words in $F^{+}_{2}$.}
\end{proposition}
 Before continuing further, we prepare some notation. For an element $w(\underline{t}) \in F^{+}_{1}$ (resp. $w(\underline{u}) \in F^{+}_{2}$), we define respectively 
\[
\widetilde{C}_{1}(w(\underline{t})\cdot s) := \left \{ \begin{array}{ll}
             C_{1}(w(\underline{t}))\cdot s & \mbox{if $t_{m}\, |_r 
\,C_{1}(w(\underline{t}))$} \\
                s & \mbox{if $t_{m} \not|_r 
\,C_{1}(w(\underline{t}))$}
          \end{array}
\right.
\]
\[
\widetilde{C}_{2}(w(\underline{u})\cdot s) := \left \{ \begin{array}{ll}
             C_{2}(w(\underline{u}))\cdot s & \mbox{if $u_{n}\, |_r 
\,C_{2}(w(\underline{u}))$} \\
                s & \mbox{if $u_{n} \not|_r 
\,C_{2}(w(\underline{u}))$}
          \end{array}
\right.
\]
For the element $\widetilde{C}_{1}(w(\underline{t})\cdot s)$ (resp. $\widetilde{C}_{2}(w(\underline{u})\cdot s)$ ), the quotient can be uniquely determined respectively. We denote it by $\widetilde{R}_{1}(w(\underline{t})\cdot s)$ (resp. $\widetilde{R}_{2}(w(\underline{u})\cdot s)$). For an arbitrary element $w$ in the monoid $G^{+}_{m, n}$, we define the set
\[
L(w) := \{ l \in  L_{0} \mid l \,|_l w \}.
\]
\begin{lemma}
{\it  Let $t_{i}$ be a letter in $L_1$. Let $w_{0}(\underline{t})$ be an element in $F^{+}_{1, \mathrm{rm}}$ that satisfies $t_{i} \not|_l w_{0}(\underline{t})$, and let $w_{0}(\underline{u})$ be an element in $F^{+}_{2, \mathrm{rm}}$.\\
Then, for an equation $t_{i} \cdot X \deeq w_{0}(\underline{t})\cdot  w_{0}(\underline{u}) \cdot s \cdot Y$, there exists a positive word $Z$ such that
\[
 X \deeq w_{0}(\underline{u}) \cdot \Delta_{1, t_{i}} \cdot \widetilde{R}_{1}(w_{0}(\underline{t})\cdot s) \cdot Z\,\, \mbox{and}\,\, Y \deeq \Delta_{1, \widetilde{C}_{1}(w_{0}(\underline{t})\cdot s)} \cdot Z.
\]}
\end{lemma}
\begin{proof}
Due to the Lemma 4.1, there exist $w_{1}(\underline{u})$ in $F^{+}_{2, \mathrm{rm}}$ and a word $Z_1$ such that 
\[
X \deeq w_{0}(\underline{u}) \cdot w_{1}(\underline{u}) \cdot \Delta_{1, t_{i}} \cdot R_{1}(w_{0}(\underline{t})) \cdot Z_{1}\,\, \mbox{and}\,\, s \cdot Y \deeq w_{1}(\underline{u}) \cdot \Delta_{1, C_{1}(w_{0}(\underline{t}))} \cdot Z_{1}.
\]
We consider the case $w_{1}(\underline{u}) \not= \varepsilon$. Then, there exists a word $Z_2$ such that $Y \deeq \Delta_{2, s} \cdot R_{2}(w_{1}(\underline{u})) \cdot Z_2$. Thus, we have an equation
\[
 \Delta_{2} \cdot R_{2}(w_{1}(\underline{u})) \cdot Z_{2} \deeq w_{1}(\underline{u}) \cdot \Delta_{1, C_{1}(w_{0}(\underline{t}))} \cdot Z_{1}.
\]
By deviding the common facter $w_{1}(\underline{u})$ from the left, we have an equation
\begin{equation}
\Delta_{2, C_{2}(w_{1}(\underline{u}))} \cdot Z_{2} \deeq \Delta_{1, C_{1}(w_{0}(\underline{t}))} \cdot Z_{1}.
\end{equation}
We can find a general solution of the equation (4.1)
\[
Z_1 \deeq  \Delta_{2, s} \cdot C_{1}(w_{0}(\underline{t})) \cdot Z_3,\,\,Z_2 \deeq  \Delta_{1, s} \cdot C_{2}(w_{1}(\underline{u})) \cdot Z_4.
\]
Thus, we have
\[
X \deeq w_{0}(\underline{u})\cdot w_{1}(\underline{u})\cdot \Delta_{1, t_{i}} \cdot R_{1}(w_{0}(\underline{t})) \cdot \Delta_{2, s} \cdot C_{1}(w_{0}(\underline{t})) \cdot Z_3 \,\,\,\,\,\,\,\,\,\,\,\,\,\,\,\,\,\,\,\,\,
\]
\[
\deeq w_{0}(\underline{u})\cdot \Delta_{1, t_{i}} \cdot w_{0}(\underline{t}) \cdot \Delta_{2, s} \cdot w_{1}(\underline{u}),\,\,\,\,\,\,\,\,\,\,\,\,\,\,\,\,\,\,\,\,\,\,\,\,\,\,\,\,\,\,\,\,\,\,\,\,\,\,\,\,\,\,\,\,\,\,\,\,\,\,\,\,\,\,\,\,\,\,\,\,\,\,\,\,\,\,
\]
\[
Y \deeq \Delta_{2, s} \cdot R_{2}(w_{1}(\underline{u})) \cdot \Delta_{1, s} \cdot C_{2}(w_{1}(\underline{u})) \cdot Z_4 \deeq \Delta_{1, s} \cdot \Delta_{2, s} \cdot w_{1}(\underline{u}).
\]
This completes the proof.
\end{proof}
\begin{lemma}
{\it  Let $t_{i}$ be a letter in $L_1$, and let $w$ be an element in $\mathcal{W}_{m, n}$ that contains at least one letter $s$ and satisfies $t_{i} \not|_l w$.We suppose that the element $w$ has the normal form: $w_{0}(\underline{t})\cdot w_{0}(\underline{u}) \cdot s \cdot w_{1}(\underline{t})\cdot w_{1}(\underline{u}) \cdot s \cdots s \cdot w_{N}(\underline{t})\cdot w_{N}(\underline{u}) \cdot s$.\\
Then, for an equation $t_{i} \cdot X \deeq w \cdot Y$, there exists a positive word $Z$ such that
\[
 X \deeq w_{0}(\underline{u}) \cdot \Delta_{1, t_{i}} \cdot w_{0}(\underline{t}) \cdot w_{1}(\underline{u}) \cdot s \cdots s \cdot w_{N-1}(\underline{t}) \cdot w_{N}(\underline{u}) \cdot s \cdot \widetilde{R}_{1}(w_{N}(\underline{t})s) \cdot Z, 
\]
\[
Y \deeq \Delta_{1, \widetilde{C}_{1}(w_{N}(\underline{t})\cdot s)} \cdot Z.\,\,\,\,\,\,\,\,\,\,\,\,\,\,\,\,\,\,\,\,\,\,\,\,\,\,\,\,\,\,\,\,\,\,\,\,\,\,\,\,\,\,\,\,\,\,\,\,\,\,\,\,\,\,\,\,\,\,\,\,\,\,\,\,\,\,\,\,\,\,\,\,\,\,\,\,\,\,\,\,\,\,\,\,\,\,\,\,\,\,\,\,\,\,\,\,\,\,\,\,\,\,\,\,\,\,\,\,\,\,\,\,\,\,\,\,\,\,\,\,\,\,\,\,\,\,\,\,\,\,\,\,\,\,\,\,\,\,\,\,\,\,\,
\]
}
\end{lemma}
\begin{proof}
By applying the Lemma 4.6 to the equation $t_{i} \cdot X \deeq w \cdot Y$ repeatedly, we say that the element $\widetilde{C}_{1}(w_{N}(\underline{t})\cdot s)$ devides $Y$ from the left. Then, we write $Y \deeq \Delta_{1, \widetilde{C}_{1}(w_{N}(\underline{t})\cdot s)} \cdot Z$ by
 some $Z$. By substituting this for the equation $t_{i} \cdot X \deeq w \cdot Y$, we have the result. This completes the proof.
\end{proof}

\begin{lemma}
{\it  Let $t_{i}$ be a letter in $L_1$, and let $w$ be an element in $\mathcal{W}_{m, n}$ that contains at least one letter $s$ and satisfies $t_{i} \not|_l w$.We suppose that the element $w$ has the normal form: $w_{0}(\underline{t})\cdot w_{0}(\underline{u}) \cdot s \cdot w_{1}(\underline{t})\cdot w_{1}(\underline{u}) \cdot s \cdots s \cdot w_{N}(\underline{t})\cdot w_{N}(\underline{u})$.\\
Then, $\mathrm{mcm}_{r}(\{ t_{i},  w \})$
\[
 = \left \{ \begin{array}{ll}
       \{ w \cdot  \Delta_{1, \widetilde{C}_{1}(w_{N-1}(\underline{t})\cdot s)\cdot w_{N}(\underline{t})}  \} & \mbox{if $w_{N}(\underline{t})\, |_l \,\Delta_{1, \widetilde{C}_{1}(w_{N-1}(\underline{t})\cdot s)}$} \\
       \{ w \cdot w'(\underline{u})\cdot \Delta_{1, C_{1}(w_{N}(\underline{t}))} \mid w'(\underline{u})\in F^{+}_{2, \mathrm{rm}} \}   & \mbox{if $w_{N}(\underline{t})\, \not|_l \,\Delta_{1, \widetilde{C}_{1}(w_{N-1}(\underline{t})\cdot s)}$}
          \end{array}
\right.
\]
}
\end{lemma}
\begin{proof}
We put $w' \deeq w_{0}(\underline{t})\cdot w_{0}(\underline{u}) \cdot s \cdots s \cdot w_{N-1}(\underline{t})\cdot w_{N-1}(\underline{u})\cdot s$. We consider an equation $t_{i} \cdot X \deeq w' \cdot Y$. Due to the Lemma 4.7, we say there exists a word $Z_1$ such that 
\[
Y \deeq \Delta_{1, \widetilde{C}_{1}(w_{N-1}(\underline{t})\cdot s)} \cdot Z_1.
\] 
We consider an equation 
\[
w_{N}(\underline{t})\cdot w_{N}(\underline{u})\cdot Y' \deeq \Delta_{1, \widetilde{C}_{1}(w_{N-1}(\underline{t})\cdot s)} \cdot Z_1.
\]
If $w_{N}(\underline{t})$ devides the element $\Delta_{1, \widetilde{C}_{1}(w_{N-1}(\underline{t})\cdot s)}$ from the left, then due to the Lemma 4.1 we say that there exists a word $Z_2$ such that
\[
Y' \deeq  \Delta_{1, \widetilde{C}_{1}(w_{N-1}(\underline{t})\cdot s)\cdot w_{N}(\underline{t})} \cdot Z_2.
\]
If $w_{N}(\underline{t})$ does not devide the element $\Delta_{1, \widetilde{C}_{1}(w_{N-1}(\underline{t})\cdot s)}$, then due to the Lemma 4.1 we say that there exist a word $w'_{0}(\underline{u})$ in $F^{+}_{2, \mathrm{rm}}$ and a word $Z_3$ such that
\[
Y' \deeq w'(\underline{u})\cdot \Delta_{1, C_{1}(w_{N}(\underline{t}))}.
\]
This completes the proof.
\end{proof}
\begin{lemma}
{\it  Let $u_{i}$ be a letter in $L_2$, and let $w$ be an element in $\mathcal{W}_{m, n}$ that contains at least one letter $s$ and satisfies $u_{i} \not|_l w$.We suppose that the element $w$ has the normal form: $w_{0}(\underline{t})\cdot w_{0}(\underline{u}) \cdot s \cdot w_{1}(\underline{t})\cdot w_{1}(\underline{u}) \cdot s \cdots s \cdot w_{N}(\underline{t})\cdot w_{N}(\underline{u})$.\\
Then, $\mathrm{mcm}_{r}(\{ u_{i},  w \})$
\[
 = \left \{ \begin{array}{ll}
       \{ w \cdot  \Delta_{2, \widetilde{C}_{2}(w_{N-1}(\underline{u})\cdot s)\cdot w_{N}(\underline{u})}  \} & \mbox{if $w_{N}(\underline{u})\, |_l \,\Delta_{2, \widetilde{C}_{2}(w_{N-1}(\underline{u})\cdot s)}$} \\
       \{ w \cdot w'(\underline{t})\cdot \Delta_{2, C_{2}(w_{N}(\underline{u}))} \mid w'(\underline{t})\in F^{+}_{1, \mathrm{rm}} \}   & \mbox{if $w_{N}(\underline{u})\, \not|_l \,\Delta_{2, \widetilde{C}_{2}(w_{N-1}(\underline{u})\cdot s)}$}
          \end{array}
\right.
\]
}
\end{lemma}

\begin{lemma}
{\it  Let $w$ be an element in $\mathcal{W}_{m, n}$ that satisfies $s \not|_l w$. We suppose that the element $w$ has the normal form: $w_{0}(\underline{t})\cdot w_{0}(\underline{u}) \cdot s \cdot w_{1}(\underline{t})\cdot w_{1}(\underline{u}) \cdot s \cdots s \cdot w_{N}(\underline{t})\cdot w_{N}(\underline{u})$.\\
Then, $\mathrm{mcm}_{r}(\{ s,  w\cdot s \})$
\[
 = \left \{ \begin{array}{lll}
 \{ w \cdot s \cdot \Delta_{1, \widetilde{C}_{1}(w_{N}(\underline{t})\cdot s)} \} & \mbox{if $w_{0}(\underline{t})\not=\varepsilon ,\, w_{0}(\underline{u})=\varepsilon$}  \\
\{ w \cdot s\cdot \Delta_{2, \widetilde{C}_{2}(w_{N}(\underline{u})\cdot s)} \}  & \mbox{if  $w_{0}(\underline{t})=\varepsilon ,\, w_{0}(\underline{u})\not=\varepsilon$} \\
       \{ w \cdot  s\cdot \Delta_{1, \widetilde{C}_{1}(w_{N}(\underline{t})\cdot s)} \cdot \Delta_{2, \widetilde{C}_{2}(w_{N}(\underline{u})\cdot s)} \} & \mbox{if $w_{0}(\underline{t}) \not=\varepsilon ,\, w_{0}(\underline{u}) \not=\varepsilon$} 
          \end{array}
\right.
\]
}
\end{lemma}
For an arbitrary element $w$ in $G^+_{m, n}$, we define a non-negative integer
\[
k(w):= \mathrm{max}\{ k \in  \Z_{\ge0} \mid \Delta^{k}\,|_l w \}.
\]
We remark that one can decide the integer $k(w)$ algorithmically. We write 
\[
w \deeq \Delta^{k(w)} \cdot w_{\mathrm{remain}}.
\]
 If $k(w) > 0$, we have $L(w) = L_{0}$. Then, we easily verify the property $(\mathrm{P}(w; \Delta))$ with respect to $w$. From here, we consider the case $k(w) = 0$. For the element $w_{\mathrm{remain}}$ and an index $i \in \{ 1, 2 \}$, we define a non-negative integer
\[
\lambda_{i}(w_{\mathrm{remain}}):= \mathrm{max}\{ k \in  \Z_{\ge0} \mid \Delta_{i}^{k}\,|_l\, w_{\mathrm{remain}} \}.
\]
By definition, we note that $\lambda_{1}(w_{\mathrm{remain}}) \cdot \lambda_{2}(w_{\mathrm{remain}}) =0$. Hence, we consider the following three cases:\par
{\bf $\mathrm{I}$}:  $\lambda_{1}(w_{\mathrm{remain}}) > 0$, $\lambda_{2}(w_{\mathrm{remain}}) = 0$.\\
For the element $w_{\mathrm{remain}}$, we define a non-negative integer
\[
\mu_{1}(w_{\mathrm{remain}}):= \mathrm{max}\{ k \in  \Z_{\ge0} \mid \Delta_{1}^{\lambda_{1}(w_{\mathrm{remain}})} \cdot \Delta^{k}_{1, s} \,|_l\, w_{\mathrm{remain}} \}.
\]
We write
\[
w \deeq w_{\mathrm{remain}} \deeq \Delta_{1}^{\lambda_{1}(w_{\mathrm{remain}})} \cdot \Delta_{1, s}^{\mu_{1}(w_{\mathrm{remain}})} \cdot w'_{\mathrm{remain}}.
\]
From the Proposition 4.4, we say that the element $w'_{\mathrm{remain}}$ belongs to the set $\mathcal{W}_{m, n}$. We take an element $A$ in $\mathrm{Trans}(w)$. By definition, there exists an element $Q$ in $G^{+}_{m, n}$ such that an equation $AQ \deeq wA$ holds. Since the element $A$ is not $\varepsilon$, there exists a letter $l_0$ in $L_0$ such that the letter $l_0$ devides $A$ from the left. We write $A \deeq l_{0} \cdot A'$. Since $L(w) = \{ s, t_1, \ldots, t_{m} \}$, we have $\{ s, t_1, \ldots, t_{m} \} \subseteq \mathrm{Trans^{min}}(w)$. We should consider the case $l_0 = u_{\tau}$ ($u_{\tau} \in L_{2}$). We consider the following two cases.\par
{\bf Case 1.}\,The element $w'_{\mathrm{remain}}$ does not contain the letter $s$.\\
We can write $w'_{\mathrm{remain}}$ by $w_{0}(\underline{t})\cdot w_{0}(\underline{u})$. We consider an equation 
\[
u_{\tau} \cdot A' \cdot Q \deeq \Delta_{1}^{\lambda_{1}(w_{\mathrm{remain}})} \cdot \Delta_{1, s}^{\mu_{1}(w_{\mathrm{remain}})} \cdot w_{0}(\underline{t})\cdot w_{0}(\underline{u}) \cdot A.
\]
From the Lemma 4.1, there exists a word $Z_1$ such that
\[
w_{0}(\underline{u}) \cdot A \deeq \Delta_{2, s} \cdot Z_1.
\]
Due to the Lemma 4.1, there exist $w'_{0}(\underline{t})$ in $F^{+}_{1, \mathrm{rm}}$ and a word $Z_2$ such that
\[
A \deeq w'_{0}(\underline{t}) \cdot \Delta_{2, C_{2}(w_{0}(\underline{u}))} \cdot Z_2.
\]
Since $\{ s, t_1, \ldots, t_{m} \} \subseteq \mathrm{Trans^{min}}(w)$, we should consider $w'_{0}(\underline{t}) = \varepsilon$. We easily show that $\{ s, t_1, \ldots, t_{m},  \Delta_{2, C_{2}(w_{0}(\underline{u}))}\} = \mathrm{Trans^{min}}(w)$. We have verified the property $(\mathrm{P}(w; \Delta))$.
\par
{\bf Case 2.}\,The element $w'_{\mathrm{remain}}$ contains the letter $s$.\\
We write $w'_{\mathrm{remain}} \deeq w_{0}(\underline{t})\cdot w_{0}(\underline{u}) \cdot s \cdot w_{1}(\underline{t})\cdot w_{1}(\underline{u}) \cdot s \cdots s \cdot w_{N}(\underline{t})\cdot w_{N}(\underline{u})$. We consider an equation 
\[
u_{\tau} \cdot A' \cdot Q \deeq \Delta_{1}^{\lambda_{1}(w_{\mathrm{remain}})} \cdot \Delta_{1, s}^{\mu_{1}(w_{\mathrm{remain}})} \cdot w'_{\mathrm{remain}}\cdot A.
\]
From the Lemma 4.1, there exists a word $Z_1$ such that
\[
w'_{\mathrm{remain}}\cdot A \deeq \Delta_{2, s} \cdot Z_1.
\]
Due to the Lemma 4.9, we consider the following two cases.\par
{\bf Case 2 -- 1.}\,$w_{N}(\underline{u})\, |_l \,\Delta_{2, \widetilde{C}_{2}(w_{N-1}(\underline{u})\cdot s)}$\\
We say that there exists a word $Z_{1}$ such that 
\[
A \deeq \Delta_{2, \widetilde{C}_{2}(w_{N-1}(\underline{u})\cdot s)\cdot w_{N}(\underline{u})} \cdot Z_{1}.
\]
We say that the element $\Delta_{2, \widetilde{C}_{2}(w_{N-1}(\underline{u})\cdot s)\cdot w_{N}(\underline{u})}$ belongs to the set $\mathrm{Trans}(w)$. Since the element $\Delta_{2, \widetilde{C}_{2}(w_{N-1}(\underline{u})\cdot s)\cdot w_{N}(\underline{u})}$ devides $\Delta$ from the left, we have verified the property $(\mathrm{P}(w; \Delta))$. \par
{\bf Case 2 -- 2.}\,$w_{N}(\underline{u})\, \not|_l \,\Delta_{2, \widetilde{C}_{2}(w_{N-1}(\underline{u})\cdot s)}$\\
We say that there exist $w'_{0}(\underline{t})$ in $F^{+}_{1, \mathrm{rm}}$ and a word $Z_1$ such that
\[
A \deeq w'_{0}(\underline{t}) \cdot \Delta_{2, C_{2}(w_{N}(\underline{u}))}\cdot Z_{1}.
\]
Since $\{ s, t_1, \ldots, t_{m} \} \subseteq \mathrm{Trans^{min}}(w)$, we should consider $w'_{0}(\underline{t}) = \varepsilon$. We say that the element $\Delta_{2, C_{2}(w_{N}(\underline{u}))}$ belongs to the set $\mathrm{Trans}(w)$. Since the element $\Delta_{2, C_{2}(w_{N}(\underline{u}))}$ devides $\Delta$ from the left, we have verified the property $(\mathrm{P}(w; \Delta))$. \par
{\bf $\mathrm{II}$}:  $\lambda_{1}(w_{\mathrm{remain}}) = 0$, $\lambda_{2}(w_{\mathrm{remain}}) > 0$.\\
We can verify the property $(\mathrm{P}(w; \Delta))$ in the same way as the case $\mathrm{I}$.\par
{\bf $\mathrm{III}$}: $\lambda_{1}(w_{\mathrm{remain}}) = \lambda_{2}(w_{\mathrm{remain}}) = 0$.\\
For the element $w_{\mathrm{remain}}$, we define a non-negative integer
\[
\mu_{1}(w_{\mathrm{remain}}):= \mathrm{max}\{ k \in  \Z_{\ge0} \mid \Delta^{k}_{1, s} \,|_l\, w_{\mathrm{remain}} \}.
\]
Next, for the element $w_{\mathrm{remain}}$, we define a non-negative integer
\[
\mu_{2}(w_{\mathrm{remain}}):= \mathrm{max}\{ k \in  \Z_{\ge0} \mid \Delta^{\mu_{1}(w_{\mathrm{remain}})}_{1, s} \cdot \Delta^{k}_{2, s}  \,|_l\, w_{\mathrm{remain}} \}.
\]
We write
\[
w \deeq w_{\mathrm{remain}} \deeq \Delta_{1, s}^{\mu_{1}(w_{\mathrm{remain}})} \cdot \Delta_{2, s}^{\mu_{2}(w_{\mathrm{remain}})} \cdot w'_{\mathrm{remain}}.
\]
From the Proposition 4.4, we say that the element $w'_{\mathrm{remain}}$ belongs to the set $\mathcal{W}_{m, n}$. We consider the following six cases.\par
{\bf Case 1.}\, $L(w) = \{ t_1, \ldots, t_{m}, u_{\tau} \}$. ($u_{\tau} \in L_{2}$)\\
The element $w'_{\mathrm{remain}}$ contains the letter $s$. We remark that the element \\$\Delta_{2, s}^{\mu_{2}(w_{\mathrm{remain}})} \cdot w'_{\mathrm{remain}}$ belongs to the set $\mathcal{W}_{m, n}$. The element $\Delta_{2, s}^{\mu_{2}(w_{\mathrm{remain}})} \cdot w'_{\mathrm{remain}}$ can be written by $w_{0}(\underline{u}) \cdot s \cdot w_{1}(\underline{t})\cdot w_{1}(\underline{u}) \cdot s \cdots s \cdot w_{N}(\underline{t})\cdot w_{N}(\underline{u})$ ($w_{0}(\underline{u}) \not=\varepsilon$). We take an element $A$ in $\mathrm{Trans}(w)$. By definition, there exists an element $Q$ in $G^{+}_{m, n}$ such that an equation $AQ \deeq wA$ holds. Since the element $A$ is not $\varepsilon$, there exists a letter $l_0$ in $L_0$ such that the letter $l_0$ devides $A$ from the left. We write $A \deeq l_{0} \cdot A'$. Since $L(w) = \{ t_1, \ldots, t_{m}, u_{\tau} \}$, we have $\{ t_1, \ldots, t_{m}, u_{\tau} \} \subseteq \mathrm{Trans^{min}}(w)$. We condider the following two cases.\par
{\bf Case 1 -- 1.}\,\,$l_{0} = s$.\\
We consider an equation
\[
s \cdot A' \cdot Q \deeq  \Delta_{1, s}^{\mu_{1}(w_{\mathrm{remain}})}\cdot \Delta_{2, s}^{\mu_{2}(w_{\mathrm{remain}})} \cdot w'_{\mathrm{remain}}\cdot s\cdot A'.
\]
Due to the Lemma 4.10, we say that there exists a word $Z_1$ such that
\[
A \deeq s \cdot \Delta_{2, \widetilde{C}_{2}(w_{N}(\underline{u})\cdot s)}\cdot Z_1.
\]
We say that the element $s \cdot \Delta_{2, \widetilde{C}_{2}(w_{N}(\underline{u})\cdot s)}$ belongs to the set $\mathrm{Trans}(w)$. Since it devides $\Delta$ from the left, we have verified the property $(\mathrm{P}(w; \Delta))$.\par
{\bf Case 1 -- 2.}\,\,$L_{0} = u_{\sigma}$ ($u_{\sigma} \not= u_{\tau}$).\\
We consider an equation 
\[
u_{\sigma} \cdot A' \cdot Q \deeq \Delta_{1, s}^{\mu_{1}(w_{\mathrm{remain}})} \cdot \Delta_{2, s}^{\mu_{2}(w_{\mathrm{remain}})}\cdot w'_{\mathrm{remain}}\cdot A.
\]
Due to the Lemma 4.9, we consider the following two cases.\par
{\bf Case 1 -- 2 -- 1.}\,\,$w_{N}(\underline{u})\, |_l \,\Delta_{2, \widetilde{C}_{2}(w_{N-1}(\underline{u})\cdot s)}$\\
We say that there exists a word $Z_{1}$ such that 
\[
A \deeq \Delta_{2, \widetilde{C}_{2}(w_{N-1}(\underline{u})\cdot s)\cdot w_{N}(\underline{u})} \cdot Z_{1}.
\]
We say that the element $\Delta_{2, \widetilde{C}_{2}(w_{N-1}(\underline{u})\cdot s)\cdot w_{N}(\underline{u})}$ belongs to the set $\mathrm{Trans}(w)$. Since the element $\Delta_{2, \widetilde{C}_{2}(w_{N-1}(\underline{u})\cdot s)\cdot w_{N}(\underline{u})}$ devides $\Delta$ from the left, we have verified the property $(\mathrm{P}(w; \Delta))$. \par
{\bf Case 1 -- 2 -- 2.}\,\,$w_{N}(\underline{u})\, \not|_l \,\Delta_{2, \widetilde{C}_{2}(w_{N-1}(\underline{u})\cdot s)}$\\
We say that there exist $w'_{0}(\underline{t})$ in $F^{+}_{1, \mathrm{rm}}$ and a word $Z_1$ such that
\[
A \deeq w'_{0}(\underline{t}) \cdot \Delta_{2, C_{2}(w_{N}(\underline{u}))}\cdot Z_{1}.
\]
Since $\{ t_1, \ldots, t_{m}, u_{\tau} \} \subseteq \mathrm{Trans^{min}}(w)$, we should consider $w'_{0}(\underline{t}) = \varepsilon$. We say that the element $\Delta_{2, C_{2}(w_{N}(\underline{u}))}$ belongs to the set $\mathrm{Trans}(w)$. Since the element $\Delta_{2, C_{2}(w_{N}(\underline{u}))}$ devides $\Delta$ from the left, we have verified the property $(\mathrm{P}(w; \Delta))$. \par
{\bf Case 2.}\, $L(w) = \{ u_1, \ldots, u_{n}, t_{\tau} \}$. ($t_{\tau} \in L_{1}$)\\
We can verify the property $(\mathrm{P}(w; \Delta))$ in the same way as the case 1.\par
{\bf Case 3.}\, $L(w) = \{ t_{\tau}, u_{\sigma} \}$. ($t_{\tau} \in L_{1}, u_{\sigma} \in L_2$)\\
We remark that the element $\Delta_{1, s}^{\mu_{1}(w_{\mathrm{remain}})}\cdot \Delta_{2, s}^{\mu_{2}(w_{\mathrm{remain}})} \cdot w'_{\mathrm{remain}}$ belongs to the set $\mathcal{W}_{m, n}$. It can be written by $w_{0}(\underline{t}) \cdot w_{0}(\underline{u}) \cdot s \cdot w_{1}(\underline{t})\cdot w_{1}(\underline{u}) \cdot s \cdots s \cdot w_{N}(\underline{t})\cdot w_{N}(\underline{u})$ ($w_{0}(\underline{t})\not=\varepsilon, w_{0}(\underline{u}) \not=\varepsilon$). We take an element $A$ in $\mathrm{Trans}(w)$. By definition, there exists an element $Q$ in $G^{+}_{m, n}$ such that an equation $AQ \deeq wA$ holds. Since the element $A$ is not $\varepsilon$, there exists a letter $l_0$ in $L_0$ such that the letter $l_0$ devides $A$ from the left. We write $A \deeq l_{0} \cdot A'$. Since $L(w) = \{ t_{\tau}, u_{\sigma} \}$, we have $\{ t_{\tau}, u_{\sigma} \} \subseteq \mathrm{Trans^{min}}(w)$. We condider the following three cases.\par
{\bf Case 3 -- 1.}\,\,$l_{0} = s$.\\
We consider an equation
\[
s \cdot A' \cdot Q \deeq  w_{\mathrm{remain}}\cdot s\cdot A'.
\]
Due to the Lemma 4.10, we say that there exists a word $Z_1$ such that
\[
A \deeq s \cdot \Delta_{1, \widetilde{C}_{1}(w_{N}(\underline{t})\cdot s)} \cdot \Delta_{2, \widetilde{C}_{2}(w_{N}(\underline{u})\cdot s)}\cdot Z_1.
\]
We say that the element $s \cdot \Delta_{1, \widetilde{C}_{1}(w_{N}(\underline{t})\cdot s)} \cdot \Delta_{2, \widetilde{C}_{2}(w_{N}(\underline{u})\cdot s)}$ belongs to the set $\mathrm{Trans}(w)$. Since it devides $\Delta$ from the left, we have verified the property $(\mathrm{P}(w; \Delta))$.\par
{\bf Case 3 -- 2.}\,\,$l_{0} = t_{\tau'}$ ($t_{\tau'} \not= t_{\tau}$).\\
We consider an equation 
\[
t_{\tau'} \cdot A' \cdot Q \deeq w_{\mathrm{remain}}\cdot A.
\]
Due to the Lemma 4.8, we consider the following two cases.\par
{\bf Case 3 -- 2 -- 1.}\,$w_{N}(\underline{t})\, |_l \,\Delta_{1, \widetilde{C}_{1}(w_{N-1}(\underline{t})\cdot s)}$\\
We say that there exists a word $Z_{1}$ such that 
\[
A \deeq \Delta_{1, \widetilde{C}_{1}(w_{N-1}(\underline{t})\cdot s)\cdot w_{N}(\underline{t})} \cdot Z_{1}.
\]
We say that the element $\Delta_{1, \widetilde{C}_{1}(w_{N-1}(\underline{t})\cdot s)\cdot w_{N}(\underline{t})}$ belongs to the set $\mathrm{Trans}(w)$. Since it devides $\Delta$ from the left, we have verified the property $(\mathrm{P}(w; \Delta))$. \par
{\bf Case 3 -- 2 -- 2.}\,$w_{N}(\underline{t})\, \not|_l \,\Delta_{1, \widetilde{C}_{1}(w_{N-1}(\underline{t})\cdot s)}$\\
We say that there exist $w'_{0}(\underline{u})$ in $F^{+}_{2, \mathrm{rm}}$ and a word $Z_1$ such that
\[
A \deeq w'_{0}(\underline{u}) \cdot \Delta_{1, C_{1}(w_{N}(\underline{t}))}\cdot Z_{1}.
\]
Then, we consider an equation 
\[
w'_{0}(\underline{u}) \cdot \Delta_{1, C_{1}(w_{N}(\underline{t}))}\cdot Z_{1} \cdot Q \deeq w_{\mathrm{remain}}\cdot w'_{0}(\underline{u}) \cdot \Delta_{1, C_{1}(w_{N}(\underline{t}))}\cdot Z_{1}.
\]
We have to consider the following two cases. \par
{\bf Case 3 - 2 - 2 - 1.}\,$w'_{0}(\underline{u}) =\varepsilon$.\\
Then, we have
\[
\Delta_{1, C_{1}(w_{N}(\underline{t}))}\cdot Z_{1} \cdot Q \deeq w_{\mathrm{remain}}\cdot \Delta_{1, C_{1}(w_{N}(\underline{t}))}\cdot Z_{1}\,\,\,\,\,\,\,\,\,\,\,\,\,\,\,\,\,\,\,\,\,\,\,\,\,\,\,\,\,\,\,\,\,\,\,\,\,\,\,\,\,\,\,\,\,\,\,\,\,\,\,\,\,\,\,\,\,\,\,\,\,\,\,\,\,\,\,\,\,\,\,\,\,
\]
\[
\deeq w_{0}(\underline{t}) \cdot w_{0}(\underline{u}) \cdot s \cdot w_{1}(\underline{t})\cdot w_{1}(\underline{u}) \cdot s \cdots s \cdot w_{N}(\underline{t})\cdot w_{N}(\underline{u})\cdot \Delta_{1, C_{1}(w_{N}(\underline{t}))}\cdot Z_{1}\,\,\,\,\,\,\,\,
\]
\[
\deeq \Delta_{1, s} \cdot w_{0}(\underline{u}) \cdot s \cdot w_{0}(\underline{t})\cdot w_{1}(\underline{u}) \cdot s \cdots s \cdot w_{N-1}(\underline{t})\cdot w_{N}(\underline{u})\cdot s \cdot R_{1}(w_{N}(\underline{t}))\cdot Z_{1}.
\]
Since the element $w_{0}(\underline{u})$ is not $\varepsilon$, due to the Lemma 4.9, we say that there exists a word $Z_2$ such that
\[
Z_1 \deeq  \Delta_{2, \widetilde{C}_{2}(w_{N}(\underline{u})\cdot s)} \cdot Z_2.
\]
We say that the element $\Delta_{1, C_{1}(w_{N}(\underline{t}))} \cdot \Delta_{2, \widetilde{C}_{2}(w_{N}(\underline{u})\cdot s)}$ belongs to the set $\mathrm{Trans}(w)$. Since it devides $\Delta$ from the left, we have verified the property $(\mathrm{P}(w; \Delta))$.
\par
{\bf Case 3 - 2 - 2 - 2.}\,$w'_{0}(\underline{u}) \not=\varepsilon$.\\
Due to the Lemma 4.9, there exists a word $Z_2$ such that 
\[
Z_1 \deeq \Delta_{2,  \widetilde{C}_{2}(w_{N}(\underline{u})\cdot w'_{0}(\underline{u})\cdot s)}\cdot Z_2.
\]
If $\widetilde{C}_{2}(w_{N}(\underline{u})\cdot w'_{0}(\underline{u})\cdot s) \deeq \widetilde{C}_{2}(w'_{0}(\underline{u})\cdot s)$, then we have
\[
w'_{0}(\underline{u})\cdot \Delta_{1, C_{1}(w_{N}(\underline{t}))}\cdot \Delta_{2, \widetilde{C}_{2}(w'_{0}(\underline{u})\cdot s)}
\]
\[
\deeq \Delta_{1, C_{1}(w_{N}(\underline{t}))}\cdot \Delta_{2, s} \cdot \widetilde{R}_{2}(w'_{0}(\underline{u})\cdot s).
\]
We say that the element $\Delta_{1, C_{1}(w_{N}(\underline{t}))}\cdot \Delta_{2, s}$ belongs to the set $\mathrm{Trans}(w)$. Since it devides $\Delta$ from the left, we have verified the property $(\mathrm{P}(w; \Delta))$. If $\widetilde{C}_{2}(w_{N}(\underline{u})\cdot w'_{0}(\underline{u})\cdot s)$ is not $\widetilde{C}_{2}(w'_{0}(\underline{u})\cdot s)$, then we say that the element $w'_{0}(\underline{u}) \cdot \Delta_{1, C_{1}(w_{N}(\underline{t}))}\cdot \Delta_{2, \widetilde{C}_{2}(w_{N}(\underline{u})\cdot w'_{0}(\underline{u})\cdot s)}$ belongs to the set $\mathrm{Trans}(w)$. Since it devides $\Delta$ from the left, we have verified the property $(\mathrm{P}(w; \Delta))$.\par
{\bf Case 3 -- 3.}\,\,$l_{0} = u_{\sigma'}$ ($u_{\sigma'} \not= u_{\sigma}$).\\
We can verify the property $(\mathrm{P}(w; \Delta))$ in the same way as the case 3 -- 2.\par
{\bf Case 4.}\, $L(w) = \{ t_{\tau} \}$. ($t_{\tau} \in L_{1}$)\\
We remark that the element $w \deeq w_{\mathrm{remain}}$ belongs to the set $\mathcal{W}_{m, n}$. It can be written by $w_{0}(\underline{t}) \cdot  s \cdot w_{1}(\underline{t})\cdot w_{1}(\underline{u}) \cdot s \cdots s \cdot w_{N}(\underline{t})\cdot w_{N}(\underline{u})$ ($w_{0}(\underline{t})\not=\varepsilon$). We take an element $A$ in $\mathrm{Trans}(w)$. By definition, there exists an element $Q$ in $G^{+}_{m, n}$ such that an equation $AQ \deeq wA$ holds. Since the element $A$ is not $\varepsilon$, there exists a letter $l_0$ in $L_0$ such that the letter $l_0$ devides $A$ from the left. We write $A \deeq l_{0} \cdot A'$. We condider the following three cases.\par
{\bf Case 4 -- 1.}\,\,$l_{0} = s$.\\
Due to the Lemma 4.10, we easily verify the property $(\mathrm{P}(w; \Delta))$.\par
{\bf Case 4 -- 2.}\,\,$l_{0} = t_{\tau'}$ ($t_{\tau'} \not= t_{\tau}$).\\
We consider an equation 
\[
t_{\tau'} \cdot A' \cdot Q \deeq w_{\mathrm{remain}}\cdot A.
\]
Due to the Lemma 4.8, we consider the following two cases.\par
{\bf Case 4 -- 2 -- 1.}\,$w_{N}(\underline{t})\, |_l \,\Delta_{1, \widetilde{C}_{1}(w_{N-1}(\underline{t})\cdot s)}$\\
We say that there exists a word $Z_{1}$ such that 
\[
A \deeq \Delta_{1, \widetilde{C}_{1}(w_{N-1}(\underline{t})\cdot s)\cdot w_{N}(\underline{t})} \cdot Z_{1}.
\]
We say that the element $\Delta_{1, \widetilde{C}_{1}(w_{N-1}(\underline{t})\cdot s)\cdot w_{N}(\underline{t})}$ belongs to the set $\mathrm{Trans}(w)$. Since it devides $\Delta$ from the left, we have verified the property $(\mathrm{P}(w; \Delta))$. \par
{\bf Case 4 -- 2 -- 2.}\,$w_{N}(\underline{t})\, \not|_l \,\Delta_{1, \widetilde{C}_{1}(w_{N-1}(\underline{t})\cdot s)}$\\
We say that there exist $w'_{0}(\underline{u})$ in $F^{+}_{2, \mathrm{rm}}$ and a word $Z_1$ such that
\[
A \deeq w'_{0}(\underline{u}) \cdot \Delta_{1, C_{1}(w_{N}(\underline{t}))}\cdot Z_{1}.
\]
If $w'_{0}(\underline{u})$ is $\varepsilon$, due to the assumption that $w_{0}(\underline{u})$ is $\varepsilon$, then we easily show that the element $\Delta_{1, C_{1}(w_{N}(\underline{t}))}$ belongs to the set $\mathrm{Trans}(w)$. Otherwise, we consider an equation 
\[
w'_{0}(\underline{u}) \cdot \Delta_{1, C_{1}(w_{N}(\underline{t}))}\cdot Z_{1} \cdot Q \deeq w_{\mathrm{remain}} \cdot w'_{0}(\underline{u}) \cdot \Delta_{1, C_{1}(w_{N}(\underline{t}))}\cdot Z_{1}.
\]
Due to the Lemma 4.9, there exists a word $Z_2$ such that 
\[
Z_1 \deeq \Delta_{2,  \widetilde{C}_{2}(w_{N}(\underline{u})\cdot w'_{0}(\underline{u})\cdot s)}\cdot Z_2.
\]
We have to consider the following two cases.\par
{\bf Case 4 - 2 - 2 - 1.}\,$\widetilde{C}_{2}(w_{N}(\underline{u})\cdot w'_{0}(\underline{u})\cdot s) \deeq \widetilde{C}_{2}(w'_{0}(\underline{u})\cdot s)$\\
Then, we have
\[
w'_{0}(\underline{u})\cdot \Delta_{1, C_{1}(w_{N}(\underline{t}))}\cdot \Delta_{2, \widetilde{C}_{2}(w'_{0}(\underline{u})\cdot s)}\,\,\,\,
\]
\[
\deeq \Delta_{1, C_{1}(w_{N}(\underline{t}))}\cdot \Delta_{2, s} \cdot \widetilde{R}_{2}(w'_{0}(\underline{u})\cdot s).
\]
We say that the element $\Delta_{1, C_{1}(w_{N}(\underline{t}))}\cdot \Delta_{2, s}$ belongs to the set $\mathrm{Trans}(w)$. Since it devides $\Delta$ from the left, we have verified the property $(\mathrm{P}(w; \Delta))$.\par
{\bf Case 4 - 2 - 2 - 2.}\,$\widetilde{C}_{2}(w_{N}(\underline{u})\cdot w'_{0}(\underline{u})\cdot s)$ is not $\widetilde{C}_{2}(w'_{0}(\underline{u})\cdot s)$\\
Then, we say that the element $w'_{0}(\underline{u}) \cdot \Delta_{1, C_{1}(w_{N}(\underline{t}))}\cdot \Delta_{2, \widetilde{C}_{2}(w_{N}(\underline{u})\cdot w'_{0}(\underline{u})\cdot s)}$ belongs to the set $\mathrm{Trans}(w)$. Since it devides $\Delta$ from the left, we have verified the property $(\mathrm{P}(w; \Delta))$. \par
{\bf Case 4 -- 3.}\,\,$l_{0} = u_{\sigma'}$ ($u_{\sigma'} \not= u_{\sigma}$).\\
We consider an equation 
\[
u_{\sigma'} \cdot A' \cdot Q \deeq w_{\mathrm{remain}}\cdot A.
\]
Due to the Lemma 4.9, we consider the following two cases.\par
{\bf Case 4 -- 3 -- 1.}\,$w_{N}(\underline{u})\, |_l \,\Delta_{2, \widetilde{C}_{2}(w_{N-1}(\underline{u})\cdot s)}$\\
We say that there exists a word $Z_{1}$ such that 
\[
A \deeq \Delta_{2, \widetilde{C}_{2}(w_{N-1}(\underline{u})\cdot s)\cdot w_{N}(\underline{u})} \cdot Z_{1}.
\]
We say that the element $\Delta_{2, \widetilde{C}_{2}(w_{N-1}(\underline{u})\cdot s)\cdot w_{N}(\underline{u})}$ belongs to the set $\mathrm{Trans}(w)$. Since it devides $\Delta$ from the left, we have verified the property $(\mathrm{P}(w; \Delta))$. \par
{\bf Case 4 -- 3 -- 2.}\,$w_{N}(\underline{u})\, \not|_l \,\Delta_{2, \widetilde{C}_{2}(w_{N-1}(\underline{u})\cdot s)}$\\
We say that there exist $w'_{0}(\underline{t})$ in $F^{+}_{1, \mathrm{rm}}$ and a word $Z_1$ such that
\[
A \deeq w'_{0}(\underline{t}) \cdot \Delta_{2, C_{2}(w_{N}(\underline{u}))}\cdot Z_{1}.
\]
Then, we consider an equation 
\[
w'_{0}(\underline{t}) \cdot \Delta_{2, C_{2}(w_{N}(\underline{u}))}\cdot Z_{1} \cdot Q \deeq w_{\mathrm{remain}}\cdot w'_{0}(\underline{t}) \cdot \Delta_{2, C_{2}(w_{N}(\underline{u}))}\cdot Z_{1}.
\]
We have to consider the following two cases. \par
{\bf Case 4 - 3 - 2 - 1.}\,$w'_{0}(\underline{t}) =\varepsilon$.\\
Then, we have
\[
\Delta_{2, C_{2}(w_{N}(\underline{u}))}\cdot Z_{1} \cdot Q \deeq w_{\mathrm{remain}}\cdot \Delta_{2, C_{2}(w_{N}(\underline{u}))}\cdot Z_{1}\,\,\,\,\,\,\,\,\,\,\,\,\,\,\,\,\,\,\,\,\,\,\,\,\,\,\,\,\,\,\,\,\,\,\,\,\,\,\,\,\,\,\,\,\,\,\,\,\,\,\,\,\,\,\,\,\,\,\,\,\,\,\,\,\,\,\,\,\,\,
\]
\[
\deeq w_{0}(\underline{t}) \cdot w_{0}(\underline{u}) \cdot s \cdot w_{1}(\underline{t})\cdot w_{1}(\underline{u}) \cdot s \cdots s \cdot w_{N}(\underline{t})\cdot w_{N}(\underline{u})\cdot \Delta_{2, C_{2}(w_{N}(\underline{u}))}\cdot Z_{1}\,\,\,\,\,\,\,
\]
\[
\deeq \Delta_{2, s} \cdot w_{0}(\underline{t}) \cdot s \cdot w_{0}(\underline{u})\cdot w_{1}(\underline{t}) \cdot s \cdots s \cdot w_{N-1}(\underline{u})\cdot w_{N}(\underline{t})\cdot s \cdot R_{2}(w_{N}(\underline{u}))\cdot Z_{1}.
\]
Since the element $w_{0}(\underline{t})$ is not $\varepsilon$, due to the Lemma 4.8, we say that there exists a word $Z_2$ such that
\[
Z_1 \deeq  \Delta_{1, \widetilde{C}_{1}(w_{N}(\underline{t})\cdot s)} \cdot Z_2.
\]
We say that the element $\Delta_{2, C_{2}(w_{N}(\underline{u}))} \cdot \Delta_{1, \widetilde{C}_{1}(w_{N}(\underline{t})\cdot s)}$ belongs to the set $\mathrm{Trans}(w)$. Since it devides $\Delta$ from the left, we have verified the property $(\mathrm{P}(w; \Delta))$.\par
{\bf Case 4 - 3 - 2 - 2.}\,$w'_{0}(\underline{t})$ is not $\varepsilon$.\\
Due to the Lemma 4.8, there exists a word $Z_2$ such that 
\[
Z_1 \deeq \Delta_{1,  \widetilde{C}_{1}(w_{N}(\underline{t})\cdot w'_{0}(\underline{t})\cdot s)}\cdot Z_2.
\]
If $\widetilde{C}_{1}(w_{N}(\underline{t})\cdot w'_{0}(\underline{t})\cdot s) \deeq \widetilde{C}_{1}(w'_{0}(\underline{t})\cdot s)$, then we have
\[
w'_{0}(\underline{t})\cdot \Delta_{2, C_{2}(w_{N}(\underline{u}))}\cdot \Delta_{1, \widetilde{C}_{1}(w'_{0}(\underline{t})\cdot s)}\,\,\,
\]
\[
\deeq \Delta_{2, C_{2}(w_{N}(\underline{u}))}\cdot \Delta_{1, s} \cdot \widetilde{R}_{1}(w'_{0}(\underline{t})\cdot s).
\]
We easily show that the element $\Delta_{2, C_{2}(w_{N}(\underline{u}))}\cdot \Delta_{1, s}$ belongs to the set $\mathrm{Trans}(w)$. Since it devides $\Delta$ from the left, we have verified the property $(\mathrm{P}(w; \Delta))$. If the element $\widetilde{C}_{1}(w_{N}(\underline{t})\cdot w'_{0}(\underline{t})\cdot s)$ is not equal to $\widetilde{C}_{1}(w'_{0}(\underline{t})\cdot s)$, then we say that the element $w'_{0}(\underline{t}) \cdot \Delta_{2, C_{2}(w_{N}(\underline{u}))}\cdot \Delta_{1, \widetilde{C}_{1}(w_{N}(\underline{t})\cdot w'_{0}(\underline{t})\cdot s)}$ belongs to the set $\mathrm{Trans}(w)$. Since it devides $\Delta$ from the left, we have verified the property $(\mathrm{P}(w; \Delta))$.\par
{\bf Case 5.}\, $L(w) = \{ u_{\sigma} \}$ ($u_{\sigma} \in L_{2}$).\\
We can verify the property $(\mathrm{P}(w; \Delta))$ in the same way as the case 4.\par
{\bf Case 6.}\, $L(w) = \{ s \}$.\\
We remark that the element $w \deeq w_{\mathrm{remain}}$ belongs to the set $\mathcal{W}_{m, n}$. It can be written by $s \cdot w_{1}(\underline{t})\cdot w_{1}(\underline{u}) \cdot s \cdots s \cdot w_{N}(\underline{t})\cdot w_{N}(\underline{u})$. We take an element $A$ in $\mathrm{Trans}(w)$. By definition, there exists an element $Q$ in $G^{+}_{m, n}$ such that an equation $AQ \deeq wA$ holds. Since the element $A$ is not $\varepsilon$, there exists a letter $l_0$ in $L_0$ such that the letter $l_0$ devides $A$ from the left. We write $A \deeq l_{0} \cdot A'$. We have to condider the following two cases.\par
{\bf Case 6 -- 1.}\,\,$l_{0} = \{ t_{\tau} \}$ ($t_{\tau} \in L_{1}$).\\
We consider an equation 
\[
t_{\tau} \cdot A' \cdot Q \deeq w_{\mathrm{remain}}\cdot A.
\]
Due to the Lemma 4.8, we consider the following two cases.\par
{\bf Case 6 -- 1 -- 1.}\,$w_{N}(\underline{t})\, |_l \,\Delta_{1, \widetilde{C}_{1}(w_{N-1}(\underline{t})\cdot s)}$\\
We say that there exists a word $Z_{1}$ such that 
\[
A \deeq \Delta_{1, \widetilde{C}_{1}(w_{N-1}(\underline{t})\cdot s)\cdot w_{N}(\underline{t})} \cdot Z_{1}.
\]
We say that the element $\Delta_{1, \widetilde{C}_{1}(w_{N-1}(\underline{t})\cdot s)\cdot w_{N}(\underline{t})}$ belongs to the set $\mathrm{Trans}(w)$. Since it devides $\Delta$ from the left, we have verified the property $(\mathrm{P}(w; \Delta))$. \par
{\bf Case 6 -- 1 -- 2.}\,$w_{N}(\underline{t})\, \not|_l \,\Delta_{1, \widetilde{C}_{1}(w_{N-1}(\underline{t})\cdot s)}$\\
We say that there exist $w'_{0}(\underline{u})$ in $F^{+}_{2, \mathrm{rm}}$ and a word $Z_1$ such that
\[
A \deeq w'_{0}(\underline{u}) \cdot \Delta_{1, C_{1}(w_{N}(\underline{t}))}\cdot Z_{1}.
\]
Then, we consider an equation 
\[
w'_{0}(\underline{u}) \cdot \Delta_{1, C_{1}(w_{N}(\underline{t}))}\cdot Z_{1} \cdot Q \deeq w_{\mathrm{remain}}\cdot w'_{0}(\underline{u}) \cdot \Delta_{1, C_{1}(w_{N}(\underline{t}))}\cdot Z_{1}.
\]
If the element $w'_{0}(\underline{u})$ is $\varepsilon$, due to the assumption that $w_{0}(\underline{u})$ is $\varepsilon$, then we easily show that the element $\Delta_{1, C_{1}(w_{N}(\underline{t}))}$ belongs to the set $\mathrm{Trans}(w)$. Otherwise, we consider an equation 
\[
w'_{0}(\underline{u}) \cdot \Delta_{1, C_{1}(w_{N}(\underline{t}))}\cdot Z_{1} \cdot Q \deeq w_{\mathrm{remain}} \cdot w'_{0}(\underline{u}) \cdot \Delta_{1, C_{1}(w_{N}(\underline{t}))}\cdot Z_{1}.
\]
Due to the Lemma 4.9, there exists a word $Z_2$ such that 
\[
Z_1 \deeq \Delta_{2,  \widetilde{C}_{2}(w_{N}(\underline{u})\cdot w'_{0}(\underline{u})\cdot s)}\cdot Z_2.
\]
{\bf Case 6 - 1 - 2 - 1.}\,$\widetilde{C}_{2}(w_{N}(\underline{u})\cdot w'_{0}(\underline{u})\cdot s) \deeq \widetilde{C}_{2}(w'_{0}(\underline{u})\cdot s)$\\
Then, we have
\[
w'_{0}(\underline{u})\cdot \Delta_{1, C_{1}(w_{N}(\underline{t}))}\cdot \Delta_{2, \widetilde{C}_{2}(w'_{0}(\underline{u})\cdot s)}\,\,\,
\]
\[
\deeq \Delta_{1, C_{1}(w_{N}(\underline{t}))}\cdot \Delta_{2, s} \cdot \widetilde{R}_{2}(w'_{0}(\underline{u})\cdot s).
\]
We say that the element $\Delta_{1, C_{1}(w_{N}(\underline{t}))}\cdot \Delta_{2, s}$ belongs to the set $\mathrm{Trans}(w)$. Since it devides $\Delta$ from the left, we have verified the property $(\mathrm{P}(w; \Delta))$.\par
{\bf Case 6 - 1 - 2 - 2.}\,$\widetilde{C}_{2}(w_{N}(\underline{u})\cdot w'_{0}(\underline{u})\cdot s)$ is not $\widetilde{C}_{2}(w'_{0}(\underline{u})\cdot s)$\\
Then, we say that the element $w'_{0}(\underline{u}) \cdot \Delta_{1, C_{1}(w_{N}(\underline{t}))}\cdot \Delta_{2, \widetilde{C}_{2}(w_{N}(\underline{u})\cdot w'_{0}(\underline{u})\cdot s)}$ belongs to the set $\mathrm{Trans}(w)$. Since it devides $\Delta$ from the left, we have verified the property $(\mathrm{P}(w; \Delta))$. \par
{\bf Case 6 -- 2.}\,\,$l_{0} = \{ u_{\sigma} \}$ ($u_{\sigma} \in L_{2}$).\\
We can verify the property $(\mathrm{P}(w; \Delta))$ in the same way as the case 6 -- 1.\par
These complete the proof. By verifying the property $(\mathrm{P}(w; \Delta))$ for an arbitrary element $w$ in the monoid $G^+_{m, n}$, we show that the conjugacy problem in it is solvable. Thus, we conclude that the conjugacy problem in the group $G_{m, n}$ is solveble. 
\begin{question}{\it In \cite{[I2]}, the author studied a positive homogeneously presented cancellative monoid $H_n^{+}$ that carries a unique minimal fundamental element $\Delta$. Thanks to the reduction lemma for it, we easily show that the monoid $H_n^{+}$ does not satisfy the LCM condition and $\mathcal{F}(H_n^{+}) = \mathcal{QZ}(H_n^{+})\setminus \{ \varepsilon \}$. Hence, the idealistic subsemigroup $\mathcal{F}(H_n^{+}) (\subseteq \mathcal{QZ}(H_n^{+}))$ is singly generated by $\Delta$ and the monoid $H_n^{+}$ is tame. Is the property $(\mathrm{P}(w; \Delta))$ satisfied for an arbitrary element $w$ in the monoid ?}
\end{question} 
\ \ \\
\emph{Acknowledgement.}\! 
The author is grateful to Kyoji Saito for fruitful discussions and valuable comments. The author is grateful to Toshitake Kohno for his encouragement. This research is supported by JSPS Fellowships for Young Scientists $(24\cdot10023)$. This researsh is also supported by World Premier International Research Center Initiative (WPI Initiative), MEXT, Japan.

\begin{flushright}
\begin{small}

Department of Mathematical Sciences, \\
University of Tokyo, \\
3-8-1 Komaba Meguro-ku Tokyo, 153-8914 Japan \\
e-mail address :  tishibe@ms.u-tokyo.ac.jp
\end{small}
\end{flushright}
\end{document}